\documentclass[10pt, english]{amsart}

\usepackage{amsmath,amssymb,enumerate}

\usepackage[T1]{fontenc}
\usepackage[all]{xy}

\usepackage{babel}
\usepackage{amstext}
\usepackage{amsmath}
\usepackage{amsfonts}
\usepackage{latexsym}
\usepackage{ifthen}

\usepackage{xypic}
\xyoption{all}
\newcommand{\rk}{{\rm rk}}

\newcommand{\sJ}{{\mathcal J}}
\newcommand{\sK}{{\mathcal K}}

\newtheorem{lemma1}{}[section]

\newenvironment{lemma}{\begin{lemma1}{\bf Lemma.}}{\end{lemma1}}
\newenvironment{example}{\begin{lemma1}{\bf Example.}\rm}{\end{lemma1}}

\newenvironment{theorem}{\begin{lemma1}{\bf Theorem.}}{\end{lemma1}}
\newenvironment{proposition}{\begin{lemma1}{\bf Proposition.}}{\end{lemma1}}
\newenvironment{corollary}{\begin{lemma1}{\bf Corollary.}}{\end{lemma1}}
\newenvironment{remark}{\begin{lemma1}{\bf Remark.}\rm}{\end{lemma1}}

\newenvironment{definition}{\begin{lemma1}{\bf Definition.}}{\end{lemma1}}
\newenvironment{notation}{\begin{lemma1}{\bf Notation.}}{\end{lemma1}}

\newenvironment{assumption}{\begin{lemma1}{\bf Assumption.}}{\end{lemma1}}

\newenvironment{remark*}{{\bf Remark.}}{}
\newenvironment{example*}{{\bf Example.}}{}
\newenvironment{assumption*}{{\bf Assumption.}}{}

\newcommand{\R}{\ensuremath{\mathbb{R}}}
\newcommand{\Q}{\ensuremath{\mathbb{Q}}}
\newcommand{\Z}{\ensuremath{\mathbb{Z}}}
\newcommand{\C}{\ensuremath{\mathbb{C}}}
\newcommand{\N}{\ensuremath{\mathbb{N}}}
\newcommand{\PP}{\ensuremath{\mathbb{P}}}

\newcommand{\Homsheaf} { \ensuremath{ \mathcal{H} \! om}}

\newcommand{\merom}[3]{\ensuremath{#1:#2 \dashrightarrow #3}}

\newcommand{\holom}[3]{\ensuremath{#1:#2  \rightarrow #3}}
\newcommand{\fibre}[2]{\ensuremath{#1^{-1} (#2)}}

\makeatletter
\ifnum\@ptsize=0  \fi \ifnum\@ptsize=2  \fi 

\newcommand\sA{{\mathcal A}}

\newcommand\sE{{\mathcal E}}

\newcommand\sF{{\mathcal F}}
\newcommand\sG{{\mathcal G}}
\newcommand\sH{{\mathcal H}}
\newcommand\sI{{\mathcal I}}
\newcommand\sQ{{\mathcal Q}}

\newcommand\sL{{\mathcal L}}

\newcommand\sO{{\mathcal O}}

\newcommand\sS{{\mathcal S}}

\newcommand\sC{{\mathcal C}}

\newcommand\bQ{{\mathbb Q}}

\newcommand\sLe{{L_{\varepsilon}}}
\newcommand\sLee{{L^3_{\varepsilon}}}
\newcommand\sLef{{L^2_{\varepsilon}}}

\DeclareMathOperator*{\pic}{Pic}
\DeclareMathOperator*{\sing}{sing}

\DeclareMathOperator*{\red}{red}

\DeclareMathOperator*{\nons}{nons}

\newcommand{\Chow}[1]{\ensuremath{\mbox{\rm Chow}(#1)}}

\DeclareMathOperator*{\supp}{Supp}

\newcommand{\NAY}{\overline{\mbox{NA}}(Y)}
\newcommand{\NEX}{\overline{\mbox{NE}}(X)}
\newcommand{\NAX}{\overline{\mbox{NA}}(X)}
\newcommand{\NE}[1]{ \ensuremath{ \overline { \mbox{NE} }(#1)} }
\newcommand{\NA}[1]{ \ensuremath{ \overline { \mbox{NA} }(#1)} }

\DeclareMathOperator*{\NS}{NS}

\title{Abundance for K\"ahler threefolds} 
\date{\today}


\author{Fr\'ed\'eric Campana}
\author{Andreas H\"oring}
\author{Thomas Peternell}

\address{Fr\'ed\'eric Campana, Institut Elie Cartan,
Universit\'e Henri Poincar\'e,
BP 239,
F-54506. Vandoeuvre-les-Nancy C\'edex,
et: Institut Universitaire de France\\
}
\email{frederic.campana@univ-lorraine.fr}

\address{Andreas H\"oring, Laboratoire de Math{\'e}matiques J.A. Dieudonn{\'e},
UMR 7351 CNRS, Universit{\'e} de Nice Sophia-Antipolis, 06108 Nice Cedex 02, France        
}
\email{hoering@unice.fr}

\address{Thomas Peternell, Mathematisches Institut, Universit\"at Bayreuth, 95440 Bayreuth, 
Germany}

\email{thomas.peternell@uni-bayreuth.de}

\begin{document}

\begin{abstract} 
Let $X$ be a compact K\"ahler threefold with terminal singularities such that $K_X$ is nef.
We prove that $K_X$ is semiample, i.e., some multiple $mK_X$ is generated by global sections. 
\end{abstract}

\maketitle


\section{Introduction}

\subsection{Main results}

Since the 1990's, the minimal model program for smooth complex projective threefolds is complete: 
every such manifold $X$ admits a birational model $X'$, which is $\mathbb Q$-factorial with only terminal singularities
such that  either 
\begin{itemize} 
\item $X'$ carries a Fano fibration, in particular, $X'$ is uniruled, or
\item the canonical bundle $K_{X'}$ is semi-ample, i.e., some positive multiple $m K_{X'}$ is generated by global sections. 
\end{itemize} 
There are basically two parts in the program: first to establish the existence of a model $X'$ which is either a Mori fibre space or has nef canonical divisor, and then to show that nefness implies semi-ampleness. This second part, known as ``abundance'', is established by 
\cite{Miy87, Kaw92c, Uta92}. 

The aim of the present paper is to fully establish the minimal model program in the category of K\"ahler threefolds. The first part of the program, i.e., the existence of a bimeromorphic model $X'$ which is a either Mori fibre space or has nef canonical divisor
was carried out in the papers \cite{HP13a} and \cite{HP15}. Thus it remains to show 
that nefness of the canonical divisor implies semi-ampleness, i.e. 
abundance holds for K\"ahler threefolds:

\begin{theorem} \label{theoremmain} 
Let $X$ be a normal $\Q$-factorial compact K\"ahler threefold with at most terminal singularities
such that $K_X$ is nef. Then $K_X$ is semi-ample, that is some positive multiple $m K_X$ is globally generated.
\end{theorem} 

The paper \cite{DP03} established the existence of some section in $mK_X$ for non-algebraic minimal models,
so the assumption above implies $\kappa(X) \geq 0$. In \cite{Pet01} abundance was shown for non-algebraic minimal models,
excluding however the case when $X$ has no non-constant meromorphic function. Our arguments do not use any information
on the structure of $X$ and work both for algebraic and non-algebraic K\"ahler threefolds.

As a corollary, we establish a longstanding conjecture \cite{Fuj83} on K\"ahler threefolds:

\begin{theorem} \label{theoremkummer} Let $X$ be smooth compact K\"ahler threefold. Assume that $X$ is simple, i.e. there is no positive-dimensional proper subvariety through a 
very general point of $X$. Then there exists a bimeromorphic morphism $X \rightarrow T/G$ where $T$ is a torus
and $G$ a finite group acting on $T$. 
\end{theorem} 

\subsection{Outline of the paper}

Let $X$ be a normal compact $\mathbb Q$-factorial 
K\"ahler threefold $X$ with only terminal singularities and nef canonical divisor $K_X$. 
Denote by $\kappa (X) $ its Kodaira dimension and  by $\nu(X)$ the numerical dimension, which is defined as
$$
\nu(X) := \max \{ m \in \N \ | \  c_1(K_X)^{m} \not \equiv 0 \}.
$$
Both invariants are subject to the inequality $\kappa (X) \leq \nu(X)$, with equality if $K_X$ is semi-ample. 
Conversely, as Kawamata observed in \cite[Thm.6.1]{Kaw85b}, in order to prove abundance, it is sufficient to prove equality: $\kappa (X) = \nu(X)$. By the base-point free theorem and an argument of Kawamata \cite[Thm.7.3]{Kaw85b} the main challenge is
to rule out the possibility
$$
\kappa(X)=0 \ \mbox{and} \ 0 < \nu(X) < 3.
$$
Since we know that $\kappa (X) \geq 0,$ there exists a positive number $m$ and an effective divisor $D$ such that 
 $D \in \vert  mK_X \vert$. A natural way to prove that $\kappa(X) \geq 1$ is to consider the restriction map
$$
r : H^0(X, d(m+1)K_X) \rightarrow H^0(D, d K_D).
$$
 Arguing by induction on the dimension we want to prove that $H^0(D, d K_D) \neq 0$ for some $d \in \N$ and that 
 some non-zero section $u \in H^0(D, d K_D)$ lifts via $r$ to a global section $\tilde u \in H^0(X, d(m+1)K_X)$ on $X$.
Since $D$ might be very singular it is however not possible to analyse the divisor $D$ directly. 
In order to circumvent this difficulty, Kawamata \cite{Kaw92c} developed the strategy, further explored in \cite{Uta92}, 
to consider log pairs $(X, B)$ with $B = \supp D$ and to improve the singularities of this pair
via certain birational transformations. 
This requires deep techniques of birational geometry of pairs within the theory of minimal models. In particular we have to run a log MMP for certain log pairs $(X, \Delta)$.
 
Therefore the first part of the paper (Sections \ref{sectionkaehler} and \ref{sectionMMP}) establishes the 
foundations for a minimal model program for log pairs on K\"ahler threefolds. As a first step we prove 
the cone theorem for the dual K\"ahler cone:
 
\begin{theorem} \label{theoremNA}
Let $X$ be a normal $\Q$-factorial compact  K\"ahler threefold that is not uniruled. 
Let $\Delta$ be an effective $\Q$-divisor on $X$ such that the pair $(X, \Delta)$ is dlt.
Then there exists an at most countable family $(\Gamma_i)_{i \in I}$  of rational curves  on $X$  
such that 
$$
0 < -(K_X+\Delta) \cdot \Gamma_i \leq 4
$$
and
$$
\NAX = \NAX_{(K_X+\Delta) \geq 0} + \sum_{i \in I} \R^+ [\Gamma_i] 
$$
\end{theorem}

The  dual K\"ahler cone $\NAX$ replaces the Mori cone of curves $\NEX$ which is obviously too small in the non-algebraic 
setting (cf. \cite[Sect.1]{HP13a}). Our result is actually more general: first we prove a weak form of the cone
theorem for lc pairs $(X, \Delta)$ such that $X$ has rational singularities (cf. Theorem \ref{theoremweakNA}).
Then we derive Theorem \ref{theoremNA} on page \pageref{proofNA} as a consequence of the weak cone theorem and some contraction results. 
The second step is to prove that the $(K_X+\Delta)$-negative extremal rays appearing in the cone theorem
can be contracted by a bimeromorphic morphism.

\begin{theorem} \label{theoremcontraction} 
Let $X$ be a normal $\Q$-factorial compact  K\"ahler threefold that is not uniruled. 
Let $\Delta$ be an effective $\Q$-divisor such that the pair $(X, \Delta)$ is dlt.
Let $R$ be a $(K_X+\Delta)$-negative extremal ray in $\NAX$.
If the ray $R$ is divisorial with exceptional divisor $S$ (cf. Definition \ref{definitionrays}), suppose that $S$ has slc singularities. 

Then the contraction of $R$ exists in the K\"ahler category.
\end{theorem}

These two theorems generalise the analogous statements for terminal threefolds \cite[Thm.1.2,Thm.1.3]{HP13a}.
While the global strategy of the proofs is similar to \cite{HP13a}, it is not an extension of earlier work: in this paper we
have to deal with threefolds with non-isolated singularities. This leads to new geometric problems (cf. the proof of 
Theorem \ref{theoremsmallray}) and is more representative of the challenges that appear in higher dimension.

Based on these results we show in Section \ref{sectionreduction} how to replace the threefold $X$
with some bimeromorphic model having a pluricanonical divisor $D \in | m K_X |$ such that $D_{\red}$ is not too singular.
In this section we follow the arguments of \cite[Ch.13, 14]{Uta92}, and we also address some technical points which do not 
appear in the literature. 
While these reduction steps assure the existence of some effective pluricanonical divisor on $D_{\red}$,
it is not obvious that it extends to a pluricanonical divisor on $X$ (cf. \cite{DHP13} for recent progress in the projective case).
Following a deformation argument of Miyaoka \cite{Miy88} \cite[Ch.11]{Uta92}, the 
case $\kappa(X)=0, \nu(X)=1$ can be excluded without too much effort. 

For the case $\nu(X)=2$ the idea is to prove via 
a Riemann-Roch computation that $h^0(X, m K_X)$ grows linearly. 
Establishing the results necessary for this Riemann-Roch computation
is the second part of the paper (Sections \ref{sectionpositivity} and \ref{sectionabundance}).
First we extend Enoki's theorem \cite{Eno88} to show that the cotangent sheaf of a non-uniruled threefold with canonical singularities is generically nef, then we want to use this positivity result to prove that $(K_X+B) \cdot \hat c_2(X) \geq 0$ where  
$K_X+B$ is a nef log-canonical divisor and $\hat c_2(X)$ is the second Chern class of the $Q$-sheaf $\hat \Omega_X$.
Since $X$ is not smooth in codimension two, the proof of this inequality in the projective case (\cite{Kaw92c}, \cite[Ch.14]{Uta92})
uses some involved computation for the second Todd class of the $Q$-sheaf $\hat \Omega_X$.
We give a new argument which should be of independent interest:
let  $X' \rightarrow X$ be a terminal modification of $X$.
Since $X'$ is smooth in codimension one, the Riemann-Roch computation on $X'$ can be done using the classical second Chern
class $c_2(X')$. Although the (log-)canonical divisor of $X'$ is not necessarily nef,
a generalisation of Miyaoka's inequality for the second Chern class (cf. Theorem \ref{theoremnonnegative}) allows us to conclude. 

{\bf Acknowledgements.} 
We thank S. Boucksom and V. Lazi\'c for some very helpful communications.
We thank the Forschergruppe 790 ``Classification of algebraic surfaces and compact complex manifolds'' of the Deutsche
Forschungsgemeinschaft for
financial support. 
A. H\"oring was partially also supported by the A.N.R. project CLASS\footnote{ANR-10-JCJC-0111}.

\section{Notation and basic facts}  \label{sectionnotation}

We will use frequently standard terminology  
of the minimal model program (MMP) as explained in \cite{KM98} or \cite{Deb01}.
We will also use without further comment results which are stated for algebraic varieties in \cite{KM98}, \cite{Uta92} if their proof
obviously works for complex spaces.

Recall that a normal complex space $X$ is $\Q$-factorial if for every Weil divisor $D$
there exists an integer $m \in \N$ such that $\sO_X(mD)$ is a locally free sheaf, i.e. $mD$ is a Cartier divisor.
Since however the canonical sheaf $K_X = \omega_X$ need not be a $\mathbb Q$-Weil divisor, we include in the 
definition of $\mathbb Q-$factoriality also the condition that there is a number $m \in \N$ such that the coherent sheaf
$$
(K_X^{\otimes m})^{**} = (\omega_X^{\otimes m})^{**}
$$
is locally free. We shall write 
$$  mK_X = (K_X^{\otimes m})^{**}$$
for short.

\begin{definition}
Let $X$ be a normal complex space. A boundary divisor is an effective $\Q$-divisor $\Delta = \sum_i a_i \Delta_i$
such that $0 \leq a_i \leq 1$.
\end{definition}

Given a boundary divisor $\Delta,$ we refer to $(X,\Delta)$ as a ``pair''. For pairs $(X,\Delta)$, the notions ``lc'' (log-canonical), ``klt'' (Kawamata log-terminal),``dlt'' (divisorial log-terminal) and ``terminal'' are defined exactly as in the algebraic context; we refer to \cite{KM98}.

\begin{definition} \label{definitionmmp}
Let $X$ be a normal $\Q$-factorial compact K\"ahler space that is not uniruled, and let $\Delta$ be a boundary divisor on $X$ such that the pair $(X, \Delta)$ is lc 
(resp. dlt, klt, terminal).
Set $(X_0, \Delta_0):=(X, \Delta)$. 
A terminating $(K+\Delta)$-MMP is a finite sequence of bimeromorphic maps
$$
\merom{\varphi_i}{(X_i, \Delta_i)}{(X_{i+1}, \Delta_{i+1}:= (\varphi_i)_* \Delta_i)}
$$
where $i \in \{ 0, \ldots, n \}$ for some $n \in \N$ that has
the following properties:
\begin{enumerate}
\item For every $i \in \{ 0, \ldots, n+1 \}$, the 
complex space $X_i$ is a normal $\Q$-factorial compact K\"ahler space
and the pair $(X_i, \Delta_i)$ is lc
(resp. dlt, klt, terminal).
\item For every $i \in \{ 0, \ldots, n \}$, the map $\varphi_i$ is either 
the contraction of a $(K_{X_i}+\Delta_i)$-negative extremal ray $R_i \in \NA{X_i}$ that is divisorial
or the flip
of a $K_{X_i}+\Delta_i$-negative extremal ray $R_i \in \NA{X_i}$
that is small.
\item The class $K_{X_{n+1}}+\Delta_{n+1}$ is nef.
\end{enumerate}
We will abbreviate such a $(K+\Delta)$-MMP by
$$
(X, \Delta) \dashrightarrow (X', \Delta')
$$
where $(X', \Delta'):=(X_{n+1}, \Delta_{n+1})$ and the bimeromorphic map
is the composition $\varphi_n \circ \ldots \circ \varphi_0$.
\end{definition}

In the definition, $\NAX $ denotes the  ``dual K\"ahler cone'', as defined in \cite[3.8]{HP13a}. We will review this notion in Definition \ref{defNA1} below. 

\subsection{Bimeromorphic models} \label{subsectionbimeromorphic}

Bimeromorphic models arising as partial resolutions of singularities play an 
important role in the recent development of the MMP for projective manifolds.
The existence of these models is usually - in the algebraic setting, see \cite[10.4]{Fuj11}, \cite[3.1]{KK10}, \cite[Ch.1.4]{Kol13} - derived as a consequence of the existence and termination of some MMP for pairs, which is at this point not yet established 
in the K\"ahler category. 
We may nevertheless assume their existence since they are constructed by taking first
a log-resolution $Y' \rightarrow X$ of the pair $(X, \Delta)$
and then running some MMP {\em over $X$}. Since we may choose the log-resolution to be a projective morphism, 
the usual relative versions of the cone and contraction theorem \cite{Nak85} \cite[Ch.3.6]{KM98} together with termination results in dimension three allows to conclude. 
The K\"ahler property of $Y$ follows again from \cite[1.3.1]{Var89}, since $\mu$ is  a projective morphism. 
We will now apply this argument in two situations:

\begin{theorem} \label{theoremterminalmodel}
Let $X$ be a normal compact K\"ahler threefold with lc singularities, i.e., $(X,0)$ is lc. 
Then there exists a bimeromorphic morphism $\mu: Y \to X$ 
from a normal $\Q$-factorial K\"ahler threefold $Y$ with terminal singularities
and a  boundary divisor $\Delta_Y$ such that the pair $(Y, \Delta_Y)$ is lc and
$$
K_Y+\Delta_Y \sim_\Q \mu^* K_X.
$$ 
\end{theorem}

\begin{proof}
Let $\holom{\pi}{\hat X}{X}$ be a resolution of singularities. Let $\hat X \dashrightarrow Y$ be a
$K_\bullet$-MMP over $X$.
The outcome of this MMP is a normal $\Q$-factorial threefold $\holom{\mu}{Y}{X}$ with terminal singularities
such that $K_Y$ is $\mu$-nef. We write
$$
K_Y + \Delta_Y \sim_\Q \mu^* K_X,
$$ 
where the support of $\Delta_Y$ is contained in the $\mu$-exceptional locus. By the negativity lemma \cite[Lemma 3.39]{KM98} we see that $\Delta_Y$ is
effective and since $X$ has lc singularities, $\Delta_Y$ is a boundary divisor. The pair $(Y, \Delta_Y)$ is lc
since this property is invariant under crepant bimeromorphic morphisms \cite[Lemma 2.30]{KM98}.
\end{proof}

\begin{theorem} \label{theoremdltmodel}
Let $X$ be a normal compact K\"ahler threefold, and let $\Delta$
be a boundary divisor on $X$ such that the
pair $(X, \Delta)$ is lc. Then $(X, \Delta)$ has a dlt model, i.e.
there exists a bimeromorphic morphism
from a normal $\Q$-factorial K\"ahler threefold $Y$ and a 
boundary divisor $\Delta_Y$ such that the pair $(Y, \Delta_Y)$ is dlt and
$$
K_Y+\Delta_Y \sim_\Q \mu^* (K_X+\Delta).
$$
\end{theorem}

\begin{proof} The same reasons as in Theorem \ref{theoremterminalmodel} applies also here; we refer to \cite[3.1]{KK10} and \cite[10.4]{Fuj11}, based on \cite[Thm.1.2]{BCHM10}, which is a result
on projective morphisms. The deduction from \cite{BCHM10} is basically a computation of discrepancies including an application of the negativity lemma. 
\end{proof}

\subsection{Adjunction} \label{subsectionadjunction}

Let $X$ be a normal complex space and let $S \subset X$ be a prime divisor that is $\Q$-Cartier. 
Let $B$ be an effective $\Q$-divisor such that $K_X+B$ is $\Q$-Cartier and $S \not\subset \supp(B)$.
Let $\holom{\nu}{\tilde S}{S}$ be the normalisation, then Shokurov's differend \cite[Ch.3]{Sho92} 
is a naturally defined effective $\Q$-divisor $B_{\tilde S}$ on $\tilde S$ 
such that
\begin{equation} \label{differend}
K_{\tilde S}+B_{\tilde S} \sim_\Q \nu^* (K_X+S+B)|_S, 
\end{equation}
From the construction in \cite[Ch.3]{Sho92} one sees immediately that 
$$
\supp \nu_*(B_{\tilde S}) \subset S_{\sing} \cup (\supp B \cap S) \cup (X_{\sing} \cap S).
$$
Suppose now that $\dim X=3$ and let $\holom{\mu}{\hat S}{\tilde S}$ be the minimal resolution.
Then we have
$$
K_{\hat S} \sim_\Q \mu^* K_{\tilde S} - N,
$$
where $N$ is an effective $\mu$-exceptional divisor \cite[4.1]{Sak81}. Set $E:=N+\mu^* B_{\tilde S}$,
then $E$ is a canonically defined effective $\Q$-divisor, such that
\begin{equation} \label{equationresolve}
K_{\hat S} + E \sim_\Q \pi^* (K_X+S+B)|_S,
\end{equation}
where  \holom{\pi}{\hat S}{S} is the composition $\nu \circ \mu$. Since $E$ is $\mu$-exceptional, we obtain
$$ \supp \pi_*(E) \subset S_{\sing} \cup (\supp B \cap S) \cup (X_{\sing} \cap S).
$$
Let $C \subset S$ be a curve such that 
$$
C \not\subset S_{\sing} \cup (\supp B \cap S) \cup (X_{\sing} \cap S).
$$
Then the morphism $\pi$ is an isomorphism in the generic point of $C$,
and we can define the strict transform $\hat C \subset \hat S$ as the closure of $C \setminus S_{\sing}$. 
By what precedes, $\hat C \not\subset E$, thus by the projection formula
and \eqref{equationresolve}, we deduce
\begin{equation} \label{inequalitycanonical}
K_{\hat S} \cdot \hat C \leq 
(K_{\hat S} + E) \cdot \hat C = (K_X+S+B) \cdot C. 
\end{equation}

\section{Singular K\"ahler spaces} \label{sectionkaehler}

\subsection{Bott-Chern cohomology}

In this section we review very briefly - following \cite{HP13a} -  the cohomology groups that replace the N\'eron-Severi space $\NS(X) \otimes \R$ from the projective setting. For details, we refer to \cite{HP13a} and the literature given there.

We recall that $\sA_X^{p,q} $ denotes the sheaf of $\mathcal C^{\infty} -$ forms of type $(p,q)$, and $\sA_X^0$ denotes the sheaf of $\sC^{\infty}-$functions; 
$\mathcal D_X$ denotes the sheaf of distributions. 

\begin{definition}  \cite[Defn. 4.6.2]{BEG13}
Let $X$ be an irreducible reduced complex space. Let $\sH_X$ be the sheaf of real parts of holomorphic
functions multiplied\footnote{We ``twist'' the definition from \cite{BEG13} in order to get a group that injects in
$H^2(X, \R)$ rather than $H^2(X, i \R)$} with $i$. A $(1,1)$-form (resp. $(1,1)$-current) with local potentials on $X$ is a global section
of the quotient sheaf $\sA^0_X/\sH_X$ (resp. $\mathcal D_X/\sH_X$). We define the Bott-Chern cohomology
$$
N^1(X) := H^{1,1}_{\rm BC}(X) := H^1(X, \sH_X).
$$
\end{definition}

Using the exact sequence 
$$
0 \rightarrow \sH_X \rightarrow \sA^0_X \rightarrow \sA^0_X/\sH_X \rightarrow 0,
$$
and the fact that $\sA^0_X$ is acyclic, we obtain a surjective map
$$
H^0(X,  \sA^0_X/\sH_X) \rightarrow H^1(X, \sH_X).
$$
Thus we can see an element of the Bott-Chern cohomology group as a closed $(1,1)$-form with local potentials
modulo all the forms that are globally of the form $dd^c u$.  Using the exact sequence
$$
0 \rightarrow \sH_X \rightarrow \mathcal D_X \rightarrow \mathcal D_X/\sH_X \rightarrow 0,
$$
we see that one obtains the same Bott-Chern group, if we consider $(1,1)$-currents with local potentials.

Recall 
the following  basic lemma, cp. \cite[Lemma 3.3.]{HP13a}.

\begin{lemma} \label{pullback} 
Let  \holom{\varphi}{X}{Y} be a proper bimeromorphic morphism between
compact normal spaces in class $\mathcal C$ with at most rational singularities. 
Then we have an injection
$$
\varphi^*: H^1(Y, \sH_Y) \hookrightarrow H^1(X, \sH_X) 
$$
and
$$
\mbox{\rm Im} \  \varphi^* = \{
\alpha \in H^1(X, \sH_X) \ | \ \alpha \cdot C = 0 \quad \forall \ C \subset X \ \mbox{curve s.t.} \ \varphi(C)=pt
\}. $$
Furthermore, let $\alpha \in H^1(X, \sH_X) \subset H^2(X, \R)$ be a class such that $\alpha = \varphi^* \beta$ with
$\beta \in H^2(Y, \R)$. Then there exists a smooth real closed $(1,1)$-form with local potentials $\omega_Y$ on $Y$
such that $\alpha = \varphi^* [\omega_Y]$.
\end{lemma}

Following \cite{Pau98}, we define the notion of a pseudo-effective resp. nef class: 

\begin{definition} \label{defPN}
Let $X$ be an irreducible reduced compact complex space and $u \in N^1(X)$.
\begin{enumerate}
\item $u$ is pseudo-effective if it can be represented by a current $T$ which is locally of the form $T = \partial \overline \partial \varphi$ with
a psh function $\varphi.$ 
\item $u$ is nef if it can be represented by a form $\alpha$ with local potentials such that for some positive $(1,1)-$form $\omega$  on $X$ and for all $\epsilon > 0,$ 
there exists $ f_{\epsilon} \in \sA^0(X),$ such that 
$$ \alpha + i \partial \overline \partial f_{\epsilon} \geq - \epsilon \omega. $$
\end{enumerate} 
\end{definition}

\begin{definition} \label{defNA1} 
Let $X$ be a normal compact complex space in class $\mathcal C.$ We define $N_1(X) $ to be the vector space of real closed currents of bidimension $(1,1)$ modulo the
 following equivalence relation: 
 $ T_1 \equiv  T_2 $ if and only if
 $$ T_1(\eta) = T_2(\eta)$$
 for all real closed $(1,1)$-forms $\eta$ (with local potentials).  \\
 We  furthermore define $\NAX \subset  N_1(X)$ to be the closed cone generated by the classes of positive closed currents. 
The closed cone of curves is the subcone
 $$ \NEX \subset \NAX $$
 of those positive closed currents arising as currents of integration over curves. 
 \end{definition}

For the proof of the contraction theorem \ref{theoremcontraction}
we will need the following statements, generalizing \cite[Prop.8.1]{HP13a}. 

\begin{proposition} \label{propositioncontraction}
Let $X$ be a normal $\Q$-factorial compact  K\"ahler space.
Let $\Delta$ be a boundary divisor such that the pair $(X, \Delta)$ is dlt.
Let $\R^+ [\Gamma_i]$ be a $(K_X+\Delta)$-negative extremal ray in $\NAX$.
Suppose that there exists a bimeromorphic morphism \holom{\varphi}{X}{Y} onto a normal complex space $Y$
such that $-(K_X+\Delta)$ is $\varphi$-ample and
a curve $C \subset X$ is contracted if and only if $[C] \in \R^+ [\Gamma_i]$.
\begin{enumerate}
\item Then we have two exact sequences
\begin{equation} \label{exactH2}
0 \rightarrow H^2(Y, \R) \stackrel{\varphi^*}{\rightarrow} H^2(X, \R) \stackrel{\alpha \mapsto \alpha \cdot \Gamma_i}{\longrightarrow} \R \rightarrow 0 
\end{equation}
\begin{equation} \label{exactN1}
0 \to N^1(Y)  \stackrel{\varphi^*}{\rightarrow}  N^1(X) \stackrel{[L] \mapsto L \cdot \Gamma_i}{\longrightarrow} \R \to 0.
\end{equation}
In particular we have $b_2(X) = b_2(Y) + 1$. 
\item We have an exact sequence
\begin{equation} \label{exactPic}
0 \rightarrow \pic(Y) \stackrel{\varphi^*}{\rightarrow} \pic(X) \stackrel{[L] \mapsto L \cdot \Gamma_i}{\longrightarrow} \Z. 
\end{equation}
\item If the contraction is divisorial, the variety $Y$
is $\Q$-factorial and its Picard number is $\rho(X)-1$.
Moreover the pair $(Y, \varphi_* \Delta)$ is dlt
\item If the contraction is small with flip $\varphi^+: X^+ \rightarrow Y$, the complex space $X^+$ 
is $\Q$-factorial and its Picard number 
is $\rho(X)$. 
Moreover the pair $(X^+, (\varphi^+)_*^{-1} \circ  \varphi_* \Delta)$ is dlt.
\end{enumerate}
\end{proposition}

\begin{proof} 
Since $-(K_X + \Delta)$ is $\varphi-$ample, the morphism $\varphi$ is projective. 
Since $(X,\Delta)$ is dlt and $X$ is $\Q$-factorial, the pair $(X,(1-\varepsilon)\Delta) $ is klt for small positive $\varepsilon$. 
Since $-(K_X + (1-\varepsilon)\Delta)$ is $\varphi-$ample, possibly after passing to a smaller $\varepsilon,$ we obtain 
$R^q\varphi_*(\sO_X) = 0$ for $q \geq 1$ by \cite[Thm.3.1]{Anc87}. Since $X$ has rational singularities (\cite[5.22]{KM98}), so does $Y$; moreover a) is proved by Lemma \ref{pullback}. 
The property $b)$ follows from \cite[Thm.3.25(4)]{KM98} and implies  the properties $c)$ and $d)$
as in \cite[Prop.3.36, Prop.3.37, Cor.3.44]{KM98}. 
\end{proof}

\begin{lemma} \label{lemmacrepantMMP}
Let $X$ be a normal $\Q$-factorial compact  K\"ahler threefold, and let $\Delta$ be a boundary
divisor such that the pair $(X, \Delta)$ is dlt. Let
$$
\varphi: X \to Y  \ {\rm resp.}  \ \merom{\varphi}{X}{X^+} 
$$
be the divisorial contraction resp. the flip of a $K_X+\Delta$-negative extremal ray $R \in \NA{X}$. Then the following holds:
\begin{enumerate}
\item Let $D$ be a nef divisor such that $D \cdot R=0$. Then $\varphi_*(D)$ is nef and
$$
\nu(D) = \nu(\varphi_*(D)).
$$
Moreover if $S \subset X$ is a prime divisor that is not contracted by $\varphi$, then $D \cdot S \neq 0$ in $H^4(X, \R)$ 
if and only if $\varphi_*(D) \cdot \varphi(S) \neq 0$ in $H^4(Y, \R)$. 
\item Let $B$ be a boundary divisor such that the pair $(X, B)$ is lc. If $(K_X+B) \cdot R=0$, then the pair $(Y, \varphi_*(B))$ is lc.
\end{enumerate}
\end{lemma}

\begin{proof} Set $L := \sO_X(D)$. We first treat the case when $\varphi $ is divisorial:
by Proposition \ref{propositioncontraction},b) we have $L \simeq \varphi^* M$ with $M$ a line bundle on $Y$ and
obviously $M \simeq \sO_Y(\varphi_* D)$. Being nef is invariant under pull-back \cite{Pau98} \cite[Lemma 3.13]{HP13a}, so $M$ is nef.
The numerical dimension is invariant since the pull-back commutes with the intersection product.
If $S' = \varphi(S)$, then $L|_S = (\varphi|_S)^*(M|_S')$, hence $L \cdot S \neq 0$ if and only  $M \cdot \varphi(S) \neq 0$ by the
the projection formula. For the property $b)$ we refer to \cite[Lemma 2.30]{KM98}. 

In the flipping case, consider the small contraction $f: X \to Y$ associated to $R$ and the flip $f^+: X^+ \to Y$. 
For $a)$ apply the considerations above first to $f$ and then to $f^+$. 
For $b)$, apply \cite[2.30]{KM98} to see that $(Y,f_*(B))$ is lc and then again to conclude that $(X^+,\varphi_*(B) = (f^+)_*^{-1} \circ f_* (B))$ is lc. 
\end{proof}

\subsection{K\"ahler criteria}

In this subsection we generalize some K\"ahler criteria given in \cite{HP13a} to threefolds with non-isolated singularities. 

\begin{theorem} \label{kaehler2} Let $X$ be a normal compact threefold
in class $\sC$.
Let $\eta \in \sA^{1,1}(X)$ be a closed real (1,1)-form with local potentials such that
$T(\eta) > 0$ for all $[T] \in \overline {NA}(X) \setminus 0.$ 
Suppose that for every irreducible curve $C \subset X$ we have $[C] \neq 0$ in $N_1(X)$.
Then $\{ \eta \} \in N^1(X)$ is represented by a K\"ahler class, in particular $X$ is K\"ahler. 
\end{theorem} 

\begin{proof} We slightly generalise the arguments of \cite[Thm.3.18]{HP13a} 
by removing the assumption that the singularities are isolated. If $X$ is smooth, 
then \cite[Thm.3.18]{HP13a} applies, so that $\{ \eta \}$ is a K\"ahler class.

We shall now reduce ourselves to this case by considering a resolution $\mu:\hat{X}\to X$ with $\hat{X}$ a compact K\"ahler manifold
such that a suitable exceptional divisor $E$ of $\hat{X}$ is $\mu$-ample. As in \cite{HP13a}, we argue that the class of an irreducible curve in $\hat X$ does not
vanish. Furthermore,
we check as in 
Step 2 of the proof of \cite[Thm.3.18]{HP13a} that 
$$T(\mu^*(\eta)-\frac{1}{n}E) > 0 $$ 
for all classes $[T] \in \overline{NA}(\hat X) $ 
if $n$ is a sufficiently big positive integer. 
Hence, by the solution in the smooth case, we may choose 
a K\"ahler form $\hat{\eta}$ on $\hat{X}$ in the class $\{\mu^*(\eta)-\frac{1}{n} E\}$.
Consider the K\"ahler current 
$$\{\mu_*(\hat{\eta})\}=\{\eta\}.$$ 
Let $Z$ be an irreducible component of the Lelong level sets of $\mu_*(\hat{\eta})$; then $Z$ is either
a point or an irreducible curve. By our assumption, $\{\eta\} \vert Z$ is a K\"ahler class on $Z$. 
By \cite[Prop.3.3(iii)]{DP04} this implies that $\{\eta\}$ is a K\"ahler class on $X$. 
\end{proof}

\begin{corollary} \label{corollarycontractionkaehler}
Let $X$ be a normal $\Q$-factorial compact  K\"ahler threefold that is not uniruled. 
Let $\Delta$ be an effective $\Q$-divisor such that the pair $(X, \Delta)$ is dlt.
Let $\R^+ [\Gamma_i]$ be a $(K_X+\Delta)$-negative extremal ray in $\NAX$.

Suppose that there exists a bimeromorphic morphism \holom{\varphi}{X}{Y} 
such that $-(K_X+\Delta)$ is $\varphi$-ample and
a curve $C \subset X$ is contracted if and only if $[C] \in \R^+ [\Gamma_i]$.
Then $Y$ is a K\"ahler space. 
\end{corollary}

\begin{proof}
As in the proof of Proposition \ref{propositioncontraction} we argue that $Y$ has rational singularities.
By assumption, $\R^+ [\Gamma_i]$ is extremal in $\NAX$; denote by $\alpha \in N^1(X)$
a nef supporting class (cf. Proposition \ref{propositionperp}). Consequently
$$
\alpha \cdot Z' > 0
$$
for every class $Z' \in \NAX \setminus \R_0^+ [\Gamma_i]$.
By Proposition \ref{propositioncontraction},a) there exists a $\beta \in N^1(Y)$ with  $\alpha = \varphi^*(\beta)$ and by what precedes we have
$$ 
\beta \cdot Z>0
$$ 
for any $Z \in \NAY \setminus 0$. \\
We claim that for any irreducible curve $C \subset Y$, there exists a curve $C' \subset X$
such that $C = \varphi(C')$. Since $C'$ is not contracted by $\varphi$ we have $\alpha \cdot C'>0$, so  
$\beta \cdot C > 0$ by the projection formula. In particular, $[C] \neq 0$ in $N_1(X)$, so the statement follows from  Theorem \ref{kaehler2}.

For the proof of the claim note first that if $C$ is not contained in the image of the exceptional locus, then we can just take the strict transform $C' \subset X$. If $C$ is in the image of the exceptional locus, then $\varphi$ is divisorial and maps a surface $S$ onto $C$.
Since the morphism $\varphi|_S: S \rightarrow C$ is projective, the surface $S$ is projective. Thus some multisection $C' \subset S$ has
the required property. 
\end{proof}

\section{MMP for pairs} \label{sectionMMP}

In this section we prove several results on the MMP for pairs for K\"ahler threefolds.  While we are not able to establish the log-MMP in full generality,
the results are sufficient for the application to the abundance problem in Section \ref{sectionreduction}.
We start by proving the cone theorem and (parts of) the contraction theorem in the Subsections \ref{subsectioncone} and \ref{subsectioncontraction}, while Subsection \ref{subsectionrunMMP}
merely collects the relevant results on existence of flips and termination.
In order to simplify the statements we make the following

\begin{assumption} \label{assumption}
Let $X$ be a normal $\Q$-factorial compact  K\"ahler threefold with rational singularities.
Let $\Delta = \sum_i a_i \Delta_i$ be an effective $\Q$-divisor on $X$ such that the pair $(X, \Delta)$ is lc.
We also assume that $X$ is not uniruled, so by \cite[Cor.1.4]{HP13a} 
the Kodaira dimension $\kappa(\hat X)$  of a desingularisation $\hat X$ is non-negative, hence $\kappa (X) \geq 0.$  Thus we have
\begin{equation} \label{canonicaldecomp}
K_X \sim_\Q \sum \lambda_j S_j,
\end{equation}
where the $S_j$ are integral surfaces in $X$, the coefficients $\lambda_j \in \Q^+$.
\end{assumption}

Note that in the situation of Theorem \ref{theoremNA} and Theorem \ref{theoremcontraction} the conditions
above are satisfied: dlt pairs have rational singularities \cite[5.22]{KM98}.

\subsection{Cone theorem} \label{subsectioncone}

In this section we will prove a weak form of the cone theorem for non-uniruled lc pairs. 
The stronger cone theorem \ref{theoremNA} will then be a consequence of the contraction results
in Subsection \ref{subsectioncontraction} which are based on the weak cone theorem:

\begin{theorem} \label{theoremweakNA}
Let $X$ be a normal $\Q$-factorial compact  K\"ahler threefold with rational singularities.
Let $\Delta$ be a boundary divisor on $X$ such that the pair $(X, \Delta)$ is lc.
If $X$ is not uniruled, there exists a positive integer 
$d \in \N$ and an at most countable family $(\Gamma_i)_{i \in I}$  of curves  on $X$
such that 
$$
0 < -(K_X+\Delta) \cdot \Gamma_i \leq d
$$
and
$$
\NAX = \NAX_{(K_X+\Delta) \geq 0} + \sum_{i \in I} \R^+ [\Gamma_i].
$$
\end{theorem}

We will prove this statement on page \pageref{proofweakNA}, but this will take some technical preparation.

\begin{lemma} \label{lemmasurfaces}
Under the Assumption \ref{assumption}, let $S \subset X$ be a prime divisor 
such that $(K_X+\Delta)|_S$ is not pseudoeffective. Then $S$ is Moishezon 
and any desingularisation $\hat S$ is a uniruled projective surface.
\end{lemma}

\begin{proof}
Let \holom{\pi}{\hat{S}}{S} be the composition of  the normalisation followed by the minimal resolution of the normalised surface.
Our goal is to show that there exists an effective $\Q$-divisor $E$ such that $K_{\hat{S}}+E$ is not pseudoeffective, in particular we have $\kappa(\hat S)=-\infty$. 
It then follows from the Castelnuovo-Kodaira classification that $\hat{S}$ is covered by rational curves,
in particular it is a projective surface \cite{BHPV04}. Thus $S$ is Moishezon.

{\em 1st case. Suppose that $K_X|_S$ is not pseudoeffective.}
Since $K_X$ is $\Q$-effective this implies that $S$ is one of the surfaces appearing in the decomposition \eqref{canonicaldecomp}.
Up to renumbering we may suppose that $S=S_1$.
Observe 
$$
S = S_1 = \frac{1}{\lambda_1} K_X -  \sum_{j=2}^r \frac{\lambda_j}{\lambda_1} S_j.
$$
Applying \eqref{equationresolve} with $B=0$, there exists an effective $\Q$-divisor $E$ such that
$$
K_{\hat S}+E
\sim_\Q  \pi^* (K_X+S)|_{S} = 
\frac{(\lambda_1+1)}{\lambda_1} \pi^* K_X|_S - 
\sum_{j=2}^r  \frac{\lambda_j}{\lambda_1}  \pi^* S_j|_S.
$$
By assumption $K_X|_S$ is not pseudoeffective, and
$- \sum_{j=2}^r  \frac{\lambda_j}{\lambda_1}  \pi^* S_j|_S$ is anti-effective. Thus
$K_{\hat S}+E$ is not pseudoeffective. 

{\em 2nd case. Suppose that $K_X|_S$ is pseudoeffective.}
Then $S \subset \Delta,$ hence, up to renumbering, we may suppose that $S=\Delta_1$. 
We claim that
$$
(K_X+\Delta_1+\sum_{i \geq 2} a_i \Delta_i)|_S 
$$
is not pseudoeffective. Otherwise, by $0 \leq a_1 \leq 1$ and the pseudo-effectivity of $K_X|_S$, the divisor
$$
(K_X+\Delta)|_S = 
a_1  (K_X+\Delta_1+\sum_{i \geq 2} a_i \Delta_i)|_S + (1-a_1) (K_X+\sum_{i \geq 2} a_i \Delta_i)|_S
$$
is a convex combination of pseudoeffective classes, hence itself pseudoeffective, a contradiction.
Applying \eqref{equationresolve} with $S= \Delta_1$ and $B=\sum_{i \geq 2} a_i \Delta_i$, there exists an effective $\Q$-divisor $E$ such that
$$
K_{\hat S}+E
\sim_\Q  \pi^* (K_X+\Delta+\sum_{i \geq 2} a_i \Delta_i)|_{S}.
$$
Thus $K_{\hat S}+E$ is not pseudoeffective.
\end{proof}

\begin{corollary} \label{corollaryalgebraicallynef}
Under the Assumption \ref{assumption}, the divisor $K_X+\Delta$
is nef if and only if
$$
(K_X+\Delta) \cdot C \geq 0 
$$
for every curve $C \subset X$.
\end{corollary}

\begin{proof}
One implication is trivial. Suppose now that $K_X+\Delta$ is nef on all curves $C$. 
We will argue by contradiction and suppose that $K_X+\Delta$ is not nef. 
Since $K_X+\Delta$ is pseudo-effective 
and the restriction to every curve is nef, there exists by \cite[Prop.3.4]{Bou04} an irreducible surface $S \subset X$ such that 
$(K_X+\Delta)|_S$ is not pseudo-effective. 
By Lemma \ref{lemmasurfaces} a desingularisation \holom{\pi}{\hat S}{S} of the surface $S$ is projective, 
so (by \cite{BDPP13}) there exists a covering family of curves $C_t \subset S$
such that for the strict transforms 
$$ \pi^* (K_X+\Delta)|_{S} \cdot \hat C_t < 0.
$$
Hence we obtain $(K_X + \Delta) \cdot C_t < 0$, a contradiction. 
\end{proof}

If $X$ is a projective manifold, the cone theorem is a consequence of Mori's bend-and-break technique.
We will now show that an analogue of this technique is available for threefolds that are lc pairs.

\begin{definition} \label{definitionbbproperty}
Let $X$ be a normal $\Q$-factorial compact  K\"ahler threefold,
and let $\Delta$ be a boundary divisor.
We say that $K_X+\Delta$ has the bend-and-break property if there exists a positive
number $d=d_{K_X+\Delta} \in \Q^+$ such that the following holds: given any curve $C \subset X$ such that
$$
- (K_X+\Delta) \cdot C > d,
$$
there exist non-zero effective 1-cycles $C_1$ and $C_2$ such that
$$
[C] = [C_1] + [C_2].
$$
\end{definition}

\begin{proposition} \label{propositionbb}
Under the Assumption \ref{assumption}, the divisor $K_X+\Delta$ has the bend-and-break property.
\end{proposition}

The proof of this result needs some preparation:

\begin{lemma} \label{lemmacrepant4}
Let $X$ be a normal $\Q$-factorial compact  K\"ahler threefold with rational singularities, and 
let $\Delta$ be a boundary divisor on $X$.
Let $\holom{\mu}{Y}{X}$ be a
bimeromorphic morphism from a normal $\Q$-factorial threefold $Y$ such that there exists
a boundary divisor $\Delta_Y$ satisfying
$$
K_Y+\Delta_Y \sim_\Q \mu^* (K_X+\Delta).
$$
Suppose that $K_Y+\Delta_Y$ has the bend-and-break property for some integer $d_{K_Y+\Delta_Y}$. Then $K_X+\Delta$ has the bend-and-break property.
\end{lemma}

\begin{proof}
Since $\dim X = 3,$ there are at most finitely many curve in the image of the $\mu$-exceptional locus $E$.
Hence 
$$
d := \max \{ d_{K_Y+\Delta_Y}, -(K_X+\Delta) \cdot Z \ | \ Z \ \mbox{an irreducible curve s.t.} \ Z \subset \mu(E) \}
$$
is a positive rational number.
Let now $C \subset X$ be a curve such that $$-(K_X+\Delta) \cdot C>d,$$
in particular $C \not \subset \mu(E).$ 
Thus the strict transform
$\hat C \subset Y$ is well-defined and
$$
- (K_Y+\Delta_Y) \cdot \hat C = -(K_X+\Delta) \cdot C>d \geq d_{K_Y+\Delta_Y}.
$$
Consequently, there exist effective non-zero $1$-cycles $\hat C_1$ and $\hat C_2$ on $Y$ such that
$$
[\hat C] = [\hat C_1] + [\hat C_2].
$$
We claim that we can find a decomposition such that
$\mu_* [\hat C_1]$ and $\mu_* [\hat C_2]$ are both non-zero. Since 
$$
[C] = \mu_*[\hat C] = \mu_* [\hat C_1] + \mu_* [\hat C_2]
$$
this will finish the proof.

To prove the claim, fix a K\"ahler form $\omega_Y$ on $Y$ such that $\omega_Y \cdot B \geq 1$ for every curve $B \subset Y$. 
We will prove the claim by induction on the degree $l := \omega_Y \cdot \hat C$.
The start of the induction for $l=1$ is trivial, since the class of a curve $\hat C$ with $\omega_Y \cdot \hat C=1$ does not decompose. 
For the induction step suppose that the claim holds  for every curve $B \subset Y$ with $\omega_Y \cdot B<l$ and
$- (K_Y+\Delta_Y) \cdot B > d$. 
Suppose that $l \leq \omega_Y \cdot \hat C <l+1$ and
consider the decomposition
$$
[\hat C] = [\hat C_1] + [\hat C_2].
$$
If both $\mu_* [\hat C_1]$ and $\mu_* [\hat C_2]$ are non-zero, there is nothing to prove, so suppose that (up to renumbering) $\mu_* [\hat C_2]=0$.
Since $\hat C_2$ is effective, all the irreducible components of $\hat C_2$ are contracted by $\mu$.
Now $K_Y + \Delta_Y$
is $\mu$-numerically trivial, so we get 
$$
- (K_Y+\Delta_Y) \cdot \hat C_1 = - (K_Y+\Delta_Y) \cdot \hat C > d.
$$  
Moreover since $\omega_Y \cdot \hat C_2 \geq 1$, we have
$\omega_Y \cdot \hat C_1<l$. Thus the induction hypothesis applies
to $C_1$, and since $[C] = \mu_* [\hat C] = \mu_* [\hat C_1]$ this proves the claim.
\end{proof}

\begin{lemma} \label{lemmaaddboundary}
In the situation of Assumption \ref{assumption}, suppose that $K_X$ has the bend-and-break
property for some integer $d_0$.
Then $K_X+\Delta$ has the bend-and-break property.
\end{lemma}

\begin{proof}
Since $X$ is a threefold, the set $(\supp \Delta)_{\sing} \cup X_{\sing}$ is a finite union of curves and points.
Hence 
$$
d := \max \{ 3, d_0, -(K_X+\Delta) \cdot Z \ | \ Z \ \mbox{a curve s.t.} \ Z \subset (\supp \Delta)_{\sing} \cup X_{\sing} \}
$$
is a positive rational number.
Let now $C \subset X$ be a curve such that $-(K_X+\Delta) \cdot C > d$.
If $\Delta \cdot C \geq 0$, then we have $-K_X \cdot C > d \geq d_0$, so we can apply the
bend-and-break property for $K_X$. Therefore we may suppose
$$
0 > \Delta \cdot C = \sum a_i \Delta_i \cdot C, 
$$
so, up to renumbering, we may suppose that $\Delta_1 \cdot C<0$. Since $0 \leq a_i \leq 1$ ($\Delta$ is a boundary), this implies
$$
- (K_X+\Delta_1+\sum_{i \geq 2} a_i \Delta_i) \cdot C \geq -(K_X+\Delta) \cdot C > d.
$$
Set $S:=\Delta_1$ and $B:=\sum_{i \geq 2} a_i \Delta_i$, and denote by $\holom{\pi}{\hat S}{S}$ the composition of  the normalisation and  the minimal resolution.
Note that by definition of $d$, the curve $C$
is not contained in the set $S_{\sing} \cup (\sum_{i \geq 2} \Delta_i \cap S) \cup (X_{\sing} \cap S)$. Therefore the strict transform $\hat C \subset \hat S$
is well-defined and \eqref{inequalitycanonical} yields 
$$
- K_{\hat S} \cdot \hat C \geq - (K_X+\Delta_1+\sum_{i \geq 2} a_i \Delta_i) \cdot C > d.
$$
Since $d \geq 3$, an application of \cite[Lemma 5.5.b)]{HP13a}
yields an effective 1-cycle $\sum_{k=1}^m \hat C_k$ in $\hat S$ with $m \geq 2$  
such that
$$
[\sum_{k=1}^m \hat C_k] = [\hat C],
$$
such that $K_{\hat S} \cdot \hat C_1 <0$, and $K_{\hat S} \cdot \hat C_2 <0$. 
Since $K_{\hat S}$ is $\pi_1$-nef, we conclude  $\pi_* [\hat C_1] \neq 0$ and $\pi_* [\hat C_2] \neq 0$. 
Thus 
$$
[C] =  \pi_* [\hat C] = \sum_{k=1}^m \pi_* [\hat C_k]
$$ 
and the first two terms of this sum are non-zero.
\end{proof}

\begin{proof}[Proof of Proposition \ref{propositionbb}]
We will first prove the statement under some additional assumptions, then reduce the general case
to this situation.

{\em Step 1. Suppose that $(X, \Delta)$ is a lc pair and $X$ has terminal singularities.}
By \cite[Cor.5.7]{HP13a} we know that $K_X$ has the bend-and-break property. 
Thus by Lemma \ref{lemmaaddboundary} the divisor $K_X+\Delta$ has the bend-and-break property. 

{\em Step 2. Suppose that $X$ has klt singularities, i.e., $(X,0)$ is klt.} 
By Theorem \ref{theoremterminalmodel} there exists a bimeromorphic morphism $\holom{\mu}{Y}{X}$
from a normal complex K\"ahler space $Y$ with terminal singularities and a boundary divisor $\Delta_Y$ such that $(Y, \Delta_Y)$ 
is lc and 
$$
K_Y+\Delta_Y \sim_\Q \mu^* K_X.
$$
By Step 1 and Lemma \ref{lemmacrepant4} the divisor $K_X$ has the bend-and-break property. 

{\em Step 3. Suppose that $(X, \Delta)$ is dlt.}
Then $(X,0)$ is klt \cite[2.39, 2.41]{KM98}.
By Step 2 and Lemma \ref{lemmaaddboundary} this implies that $K_X+\Delta$ has the bend-and-break property.

{\em Step 4. General case.}
By Theorem \ref{theoremdltmodel}
there exists a bimeromorphic morphism $\holom{\mu}{Y}{X}$
and a boundary divisor $\Delta_Y$ on $Y$ such that $(Y, \Delta_Y)$ 
is dlt and 
$$
K_Y+\Delta_Y \sim_\Q \mu^* (K_X+\Delta).
$$
By Step 3 and Lemma \ref{lemmacrepant4} this implies that $K_X+\Delta$ has the bend-and-break property.
\end{proof}

We can now prove the weak cone theorem for the classical Mori cone.

\begin{proposition} \label{propositionNE}
Under the Assumption \ref{assumption}
there exists a number $d \in \Q^+$ and an at most countable family $(\Gamma_i)_{i \in I}$  of curves on $X$
such that 
$$
0 < -(K_X+\Delta) \cdot \Gamma_i \leq d
$$
and
$$
\NEX = \NEX_{(K_X+\Delta) \geq 0} + \sum_{i \in I} \R^+ [\Gamma_i] 
$$
\end{proposition}

\begin{proof}
By Proposition \ref{propositionbb} there exists a positive
number $d \in \Q^+$ realising the bend-and-break
property (cf. Definition \ref{definitionbbproperty}) for $K_X+\Delta$.
Since there are only countably many curve classes $[C] \subset \NEX$,
we may choose a representative $\Gamma_i$ for each class such that  $0 < -(K_X+\Delta) \cdot \Gamma_i \leq d$.
We set
$$
V := \NEX_{(K_X+\Delta) \geq 0} + \sum_{0 < -(K_X+\Delta) \cdot \Gamma_i \leq d} \R^+ [\Gamma_i].
$$
Fix a K\"ahler form $\eta_X$ on $X$ such that 
$\eta_X \cdot C \geq 1$ 
for every curve $C \subset X$. 

We need to prove that $\NEX=V$.
By \cite[Lemma 6.1]{HP13a}\footnote{The statement in \cite[Lemma 6.1]{HP13a} is for the canonical class $K_X$, but the proof does not use this hypothesis.}
it is sufficient to show that $$\NEX = \overline V,$$  i.e.
the class $[C]$ of every irreducible curve $C \subset X$
is contained in $V$. 
We will prove the statement by induction on the degree $l:=\eta_X \cdot C$. The start
of the induction for $l=0$ being trivial,
we suppose that we have shown the statement for all curves of degree at most $l-1$
and let $C$ be a curve such that $$l-1 < \eta_X \cdot C \leq l.$$
If $-(K_X+\Delta) \cdot C \leq d$ we are done. Otherwise, there exists by the bend-and-break property 
a decomposition
$$
[C] = [C_1] + [C_2]
$$
with $C_1$ and $C_2$ non-zero effective 1-cycles on $X$.
Since $\eta_X \cdot C_i \geq 1$ for $i=1,2$ we have
$\eta_X \cdot C_i \leq l-1$ for $i=1,2$. By induction both classes are in $V$, so $[C]$ is in $V$.
\end{proof}

\begin{proof}[Proof of Theorem \ref{theoremweakNA}] \label{proofweakNA}
We follow the strategy of the proof of \cite[Prop.6.4]{HP13a}. 
By Proposition \ref{propositionNE}  there exists a number 
$d \in \Q^+$ and an at most countable family $(\Gamma_i)_{i \in I}$  of curves on $X$
such that 
$0 < -(K_X+\Delta) \cdot \Gamma_i \leq d$
and
$$
\NEX = \NEX_{(K_X+\Delta) \geq 0} + \sum_{i \in I} \R^+ [\Gamma_i]. 
$$
Set
$$
V := \NAX_{(K_X+\Delta) \geq 0} + \sum_{i \in I} \R^+ [\Gamma_i].
$$
By \cite[Lemma 6.1]{HP13a} it is sufficient to show that $\NAX \subset \overline{V}$.
Let $\pi: \hat X \to X$ be a desingularisation, then by \cite[Prop.3.14]{HP13a} we have $\NAX = \pi_* (\overline {NA}(\hat X))$. 
Thus it is sufficient to prove that for a set of generators $\hat \alpha_i$ of  $\overline {NA}(\hat X)$, 
we have $\pi_* (\alpha_i) \in \overline{V}$.
By \cite[Cor.0.3]{DP04} the cone $\overline {NA}(\hat X)$ is the closure of the convex cone generated by cohomology classes
of the form $[\hat \omega]^2, [\hat \omega] \cdot [\hat S]$ and $[\hat C]$ where $\hat \omega$ is a K\"ahler form,
$\hat S$ a surface and $\hat C$ a curve on $\hat X.$ 
Let now $\hat \alpha$ be such a generator, then
our goal is to show that the push-forward $\alpha:= \pi_* (\hat \alpha)$ of any of this three types is contained in $\overline V$. 

{\em 1st case. $\hat \alpha = [\hat \omega]^2$ with $\hat \omega$ a K\"ahler form.} 
Since
$\pi^*(K_X+\Delta)$ is pseudoeffective, we have $\pi^*(K_X+\Delta) \cdot [\hat \omega]^2 \geq 0$, hence 
$$ (K_X+\Delta) \cdot \alpha = (K_X+\Delta) \cdot \pi_*(\hat \alpha) = \pi^*(K_X+\Delta) \cdot \hat \alpha \geq 0, $$
and thus 
$\alpha \in \NAX_{(K_X+\Delta) \geq 0}$. 

{\em 2nd case. $\hat \alpha = [\hat C]$ with $\hat C$ a curve.} 
Then set $C := \pi_*(\hat C)$, so that $\alpha = [C].$
Since we have an inclusion
\begin{equation} \label{inclusion}
\NEX = \NEX_{(K_X+\Delta) \geq 0} + \sum_{i \in I} \R^+ [\Gamma_i] \hookrightarrow \NAX_{(K_X+\Delta) \geq 0} + \sum_{i \in I} \R^+ [\Gamma_i], 
\end{equation}
and by hypothesis any curve class $[C]$ is in the left hand side, we see that $[C] \in V$.\\
(3) 
{\em 3rd case. $\hat \alpha = [\hat \omega] \cdot [\hat S]$ with $\hat S$ an irreducible surface
and $\hat \omega$ a K\"ahler form.}
If 
$$\pi^*(K_X+\Delta) \cdot [\hat \omega] \cdot [\hat S] \geq 0,$$
the class $\alpha $ is in $\NAX_{(K_X+\Delta) \geq 0}$.
Suppose now that 
$$\pi^*(K_X+\Delta) \cdot [\hat \omega] \cdot [\hat S] < 0.$$ 
Using the projection formula we see that $\pi(\hat S)$ is not a point.

{\em Case a) Suppose that $\pi(\hat S)$ is a surface $S$.}
Since $\pi^*(K_X+\Delta) \cdot [\hat \omega] \cdot [\hat S] < 0$, the restriction $\pi^*(K_X+\Delta)|_{\hat S}$ is not pseudoeffective. 
Thus the restriction $(K_X+\Delta) \vert_S$ is not pseudoeffective. 
Hence $S$ is covered by rational curves by Lemma \ref{lemmasurfaces}.
Since $\hat S \rightarrow S$ is bimeromorphic, the same
property holds for $\hat S$. 

Let 
\holom{\pi}{\bar{S}}{\hat{S}} be 
the composition of the normalisation and the minimal resolution,
then $\bar S$ is a uniruled projective surface, in particular, $H^2(\bar S, \sO_{\bar S})=0$. Hence the Chern class map
$$
\pic{\bar S} \rightarrow H^2(\bar S, \Z)
$$ 
is surjective, so $\pi^* (\hat \omega|_{\hat S}),$ 
which is a real closed form of type $(1,1)$, is represented by  a $\R$-divisor which is nef and big.
As in the proof of \cite[Prop.6.4]{HP13a}, this implies that
$$ \hat \alpha = 
[\hat \omega] \cdot [\hat S] \in \NE{\hat X}.
$$
Thus we have 
$$
\alpha = \pi_* ([\hat \omega] \cdot [\hat S]) \in \NEX, 
$$
so $\alpha$
is in the image of the inclusion \eqref{inclusion}.

{\em Case b) Suppose that $\pi(\hat S)$ is a curve $C$.}
We claim that there exists a number $\lambda \in \R$ such that 
$$
\alpha = \pi_* (\hat \alpha) = \lambda [C]
$$
in $N_1(X)$, in particular $\alpha$
is in the image of the inclusion \eqref{inclusion}
which finishes the proof.
By duality it is sufficient to prove that there exists a $\lambda \in \mathbb R$ such that  for every class
$H \in N^1(X)$ we have
$$
H \cdot \pi_* (\hat \alpha) = \lambda (H \cdot C).
$$
By the projection
formula one has
$$
H \cdot \pi_* (\hat \alpha)
= \pi^* H \cdot \hat \alpha 
= \pi^* H \cdot  [\hat \omega] \cdot [\hat S]
= [(\pi^* H)|_S] \cdot [\hat \omega|_{\hat S}].
$$
By definition of $C$ we have a surjective map $\hat S \rightarrow C$, 
so $[(\pi^* H)|_S]$ is numerically equivalent to
$(H \cdot C) [F]$ where $F$ is a general fibre of $\hat S \rightarrow C$. Thus we see that
$$
[(\pi^* H)|_S] \cdot [\hat \omega|_{\hat S}]
=
\lambda (H \cdot C)
$$
where $\lambda := [\hat \omega|_{\hat S}] \cdot F$
does not depend on $H$.
\end{proof}

\subsection{The contraction theorem} \label{subsectioncontraction}

Suppose that the Assumption \ref{assumption} holds.
For the whole subsection we 
fix  $R := \R^+ [\Gamma_{i_0}]$ a $(K_X+\Delta)$-negative extremal ray in $\NAX$. 
The following proposition is a well-known consequence of the weak cone theorem \ref{theoremweakNA},
cf. \cite[Prop.7.3]{HP13a} for details:

\begin{proposition} \label{propositionperp} 
There exists a nef class $\alpha \in N^1(X)$ such that
$$ 
R = \{ z \in \overline{NA}(X) \ \vert \ \alpha \cdot z = 0\},
$$
and such that, using the notation of Theorem \ref{theoremweakNA}, the class $\alpha$ is strictly positive on
$$
\left( \NAX_{(K_X+\Delta) \geq 0} + \sum_{i \in I, i \neq i_0 } \R^+ [\Gamma_i] \right) \setminus \{0\}.
$$ 
We call $\alpha$ a nef supporting class for the extremal ray $R$.
\end{proposition} 

Note first that by hypothesis, the cohomology class $\alpha-(K_X+\Delta)$ is positive on the extremal ray $R$, moreover we know that $\alpha$ is positive on
$$
\left( \NAX_{(K_X+\Delta) \geq 0} + \sum_{i \in I, i \neq i_0 } \R^+ [\Gamma_i] \right) \setminus \{0\}.
$$ 
Thus, up to replacing $\alpha$ by some positive multiple, we can suppose that $\alpha-(K_X+\Delta)$ is
positive on $\NAX \setminus \{ 0 \}$. Since $X$ is a K\"ahler space, this implies by \cite[Cor.3.16]{HP13a} that
\begin{equation}\label{equationomega}
\omega := \alpha-(K_X+\Delta)
\end{equation}
is a K\"ahler class. Let now $\holom{\pi}{\hat X}{X}$ be a desingularisation.
Since $K_X+\Delta$ is $\Q$-effective, this also holds for $\pi^* (K_X+\Delta)$. Thus the nef class $\pi^* \alpha$
is the sum of a $\Q$-effective class $\pi^* (K_X+\Delta)$ and the nef and big (i.e. contains a K\"ahler current, see \cite[Thm. 0.5]{DP04} ) class $\pi^* \alpha$.
Thus $\pi^* \alpha$ is nef and big and 
we have
\begin{equation} \label{equationalphabig}
\alpha^3 = (\pi^* \alpha)^3 >0.
\end{equation}

We divide extremal rays into two classes, according to the deformation behaviour of the curves they contain:

\begin{definition} \label{definitionrays}
We say that the $(K_X+\Delta)$-negative extremal ray $R$ is small if every curve $C \subset X$ with
$[C] \in R$ is very rigid in the sense of \cite[Defn.4.3]{HP13a}.
Otherwise we say that the extremal ray $R$ is divisorial.
\end{definition}

\subsubsection{Small rays}

\begin{theorem} \label{theoremsmallray}
Under the Assumption \ref{assumption}, 
suppose that the extremal ray $R=\R^+ [\Gamma_{i_0}]$ is small. 
Then the contraction of the ray $R$ exists. 
\end{theorem}

The proof requires a significant refinement 
of the argument in \cite{HP13a} since the description of the non-K\"ahler locus of $\alpha$ is much more
delicate for non-isolated singularities. 
The following lemma is a key ingredient:

\begin{proposition} \label{propositionalphabig}
Under the Assumption \ref{assumption},
suppose that the extremal ray $R=\R^+ [\Gamma_i]$ is small. 
Let $S \subset X$ be an irreducible surface. Then we have
$\alpha^2 \cdot S>0$.
\end{proposition} 

At this point we cannot yet exclude the possibility that there are infinitely many distinct curves $C \subset X$ 
such that $[C] \in R$\footnote{
If $X$ has terminal singularities we can use some additional argument to obtain this property, cf. \cite[Rem.7.2]{HP13a}.}.
However by definition of a small ray no such curve (or its multiples) deforms. Since the irreducible components of the cycle space are countable 
we see that there are at most countably many curves $C \subset X$ 
such that $[C] \in R$.

\begin{proof}[Proof of Proposition \ref{propositionalphabig}]
If $\alpha|_S=0$ then \eqref{equationomega} implies that
$$
-(K_X+\Delta)|_S = \omega|_S.
$$
Thus the divisor $-(K_X+\Delta)|_S$ is ample, in particular $S$
is projective and $S$ is covered by curves. Since $\alpha|_S=0$, the classes of all these curves are in $R$,
a contradiction. 

Arguing by contradiction we suppose now that $\alpha|_S \neq 0$
but $\alpha^2 \cdot S=0$. Then we have
$$
0 = \alpha^2 \cdot S = (K_X+\Delta) \cdot \alpha \cdot S
+ \omega \cdot \alpha \cdot S 
$$ 
and 
$$
\omega \cdot \alpha \cdot S = \omega|_S \cdot \alpha|_S >0
$$ 
since $\alpha|_S$ is a non-zero nef class. 
Thus we obtain
\begin{equation} \label{help1}
(K_X+\Delta) \cdot \alpha \cdot S < 0.
\end{equation}
In particular $(K_X+\Delta)|_S$ is not pseudoeffective, the class $\alpha|_S$ being nef. 

Let \holom{\pi}{\hat S}{S} be the composition of normalisation and minimal resolution (cf. Subsection \ref{subsectionadjunction}).
We claim that there exists an effective $\Q$-divisor $E$
on $\hat S$ such that
\begin{equation} \label{equationnotgood}
(K_{\hat S}+E) \cdot \pi^* (\alpha|_S) < 0.
\end{equation}
Assuming this for the time being, let us see how to derive
a contradiction: note first that $K_{\hat S}$ is not pseudoeffective, so $\hat S$ is uniruled and projective. In particular the nef class $\pi^* (\alpha|_S)$
is represented by an $\R$-divisor. 
Fix an ample divisor $A$ on $\hat S$. 
By \cite[Thm.1.3]{Ara10} we know that given $\varepsilon>0$, there exists a decomposition
$$
\pi^* (\alpha|_S) = C_{\varepsilon} + \sum \lambda_{i, \varepsilon} M_{i, \varepsilon}
$$
where $\lambda_{i, \varepsilon} \geq 0$, $(K_{\hat S}+\varepsilon A) \cdot C_\varepsilon \geq 0$ and the $M_{i, \varepsilon}$ are movable curves. Since $M_{i, \varepsilon}$ 
belongs to an (uncountable) deformation family of curves we obtain
$$
\pi^* (\alpha|_S) \cdot M_{i, \varepsilon} = \alpha \cdot \pi_*(M_{i, \varepsilon}) > 0.
$$
Since $(\pi^* (\alpha|_S))^2=0$ this implies that $\pi^* (\alpha|_S) = C_{\varepsilon}$
for all $\varepsilon>0$. Passing to the limit $\varepsilon \to 0$ we obtain $K_{\hat S} \cdot \pi^* (\alpha|_S) \geq 0$, a contradiction to \eqref{equationnotgood}.

{\em Proof of the claim \eqref{equationnotgood}.} As in the proof
of Lemma \ref{lemmasurfaces} we need a case distinction.

{\em 1st case. Suppose that $K_X \cdot \alpha \cdot S < 0$.}
Since $K_X$ is $\Q$-effective, $S$ is one of the surfaces $S_j$ in the decomposition \eqref{canonicaldecomp}
and
$S \cdot \alpha \cdot S < 0$. 
In particular we have
\begin{equation} \label{help}
(K_X+S) \cdot \alpha \cdot S<0.
\end{equation}
Applying \eqref{equationresolve} with $B=0$ we obtain an 
effective $\Q$-divisor $E$ on $\hat S$ such that
$K_{\hat S}+E \sim_\Q \pi^* (K_X+S)|_S$. Thus 
\eqref{equationnotgood} follows from \eqref{help}.

{\em 2nd case. Suppose that $K_X \cdot \alpha \cdot S \geq 0$.}
By \eqref{help1} this implies that $\Delta \cdot \alpha \cdot S<0$, so $S$ is contained
in the support of $\Delta$; up to renumbering we may
suppose that $S=\Delta_1$. Since $\Delta_i \cdot \alpha \cdot S \geq 0$ for every $i \geq 2$, it follows $S \cdot \alpha \cdot S<0$. Since $0 \leq a_i \leq 1$ we conclude
\begin{equation} \label{help2}
(K_X+S+\sum_{i \geq 2} a_i \Delta_i) \cdot \alpha \cdot S
\leq (K_X+\Delta) \cdot \alpha \cdot S
<0.
\end{equation}
Applying \eqref{equationresolve} with $B=\sum_{i \geq 2} a_i \Delta_i$ we obtain an 
effective $\Q$-divisor $E$ on $\hat S$ such that
$K_{\hat S}+E \sim_\Q \pi^* (K_X+S+\sum_{i \geq 2} a_i \Delta_i)|_S$. Thus 
\eqref{equationnotgood} follows from \eqref{help2}.
\end{proof}

\begin{proof}[Proof of Theorem \ref{theoremsmallray}] 
Let $\pi: \hat X \to X$ be a desingularisation. By \eqref{equationalphabig} we have $(\pi^* \alpha)^3>0$,
so we can apply the theorem of Collins and Tosatti \cite[Thm.1.1]{CT13} to $\pi^* \alpha$: the
non-K\"ahler locus $E_{nK}(\pi^* \alpha)$ is equal to the null-locus, i.e. we have
$$
E_{nK}(\pi^* \alpha) = \cup_{(\pi^* \alpha)|_Z^{\dim Z}=0} Z, 
$$
where the union runs over all the subvarieties of $\hat X$.
If $Z \subset X$ is a surface such that $(\pi^* \alpha)|_Z^{2}=0$ then it follows from the projection formula
and Proposition \ref{propositionalphabig} that $\dim \pi(Z) \leq 1$.
Thus we see that $E_{nK}(\pi^* \alpha)$
is a finite union of $\pi$-exceptional surfaces and curves.
Since
$$
E_{nK}(\alpha) \subset \pi(E_{nK}(\pi^* \alpha)),
$$
and $X_{\sing}$ is a union of curves and points,
we see that $E_{nK}(\alpha)$ is a union of 
curves and points. Clearly $E_{nK}(\alpha)$ contains all the curves
$C \subset X$ such that $[C] \in R$, in particular 
$$
C := \cup_{C_l \subset X, [C_l] \in R} C_l
$$
is a finite union of curves. Our goal is to show that the connected components of $C$ can be contracted 
onto points. Since it is a-priori not clear that $C=E_{nK}(\alpha)$ we have to improve the construction from
\cite[Thm.7.12]{HP13a} 

By \cite[Prop.2.3]{Bou04}\footnote{Proposition 2.3. in \cite{Bou04} is for manifolds, but we can adapt the proof
as follows: the non-K\"ahler locus of $\pi^*\alpha$ is the union of the $\pi$-exceptional locus and the strict transforms of the curves
in the non-K\"ahler locus of $\alpha$. By \cite[Thm.3.1.24]{Bou02} there exists a modification
$\holom{\mu'}{\tilde X}{\hat X}$ and a K\"ahler class $\tilde \alpha$ on $\tilde X$ such that $(\mu')^* \pi^* \alpha= \tilde \alpha + E$
where $E$ is an effective $\R$-divisor. Then $\mu:=\pi \circ \mu'$ has the stated properties.} there exists a modification
$\holom{\mu}{\tilde X}{X}$ and a K\"ahler class $\tilde \alpha$ on $\tilde X$ such that $\mu_* \tilde \alpha= \alpha$.
Since $\tilde \alpha - \mu^* \alpha$ is $\mu$-nef, the negativity lemma \cite[Lemma 3.39]{KM98} \cite[3.6.2]{BCHM10} implies that we have
\begin{equation} \label{equationbig}
\mu^* \alpha = \tilde \alpha + E  
\end{equation}
with $E$ an effective $\mu$-exceptional $\R$-divisor. 
Let now $C_l \subset X$ be an irreducible curve such that $[C_l] \in R$, and set
$D_l$ for the support of $\fibre{\mu}{C_l}$.
Then we have $\alpha|_{C_l} \equiv 0$, so by \eqref{equationbig} we have
$$
-E|_{D_l} = \tilde \alpha|_{D_l}. 
$$
Thus $-E|_{D_l}$ is ample. If $B \subset X$ is an arbitrary curve contained in the image of the exceptional locus,
and $D_B$ the support of $\fibre{\mu}{B}$, we still have
$$
-E|_{D_B} = \tilde \alpha|_{D_B} - (\mu^* \alpha)|_{D_B}. 
$$
Thus we see that $-E$ is $\mu$-ample, in particular its support is equal to
the $\mu$-exceptional locus.
Since ampleness is an open property we can find an effective $\Q$-divisor $E' \subset \tilde X$ that is $\mu$-ample and
such that $-{E'}|_{D_l}$ is ample for all $l$. Up to taking a positive multiple we can suppose that $E'$ has integer coefficients.
We set
$$
\sK_m := \mu_* \sO_{\tilde X}(-mE'),
$$
and {\it we claim} that there exists a $m \in \N$ such that the restriction $\sK_m|_{C_l}$ is ample for all $j$.
Note that this implies that the quotient $(\sK_m/\sK_m^2)|_{C_l}$ is ample for all $j$. 

{\em Step 2. Proof of the claim.}
By relative Serre vanishing there exists an $m_0 \in \N$ such that
$$
R^i \mu_* \sO_{\tilde X}(-mE')=0 \qquad \forall \ i>0, m \geq m_0
$$
and 
\begin{equation} \label{serre}
R^i \mu_* \sO_{\tilde X}(-D_l-mE')=0 \qquad \forall \ i>0, m \geq m_0
\end{equation}
and all $l$. Moreover we know by \cite[Thm.3.1.]{Anc82} that there exists an $m_1 \in \N$ such that the direct
image sheaf
$$
(\mu|_{D_l})_* (\sO_{\tilde X}(-mE') \otimes \sO_{D_l})
$$
is ample for all $m \geq m_1$ and $l$. 

Fix now an $m \geq \max \{ m_0, m_1\}$ and consider now the exact sequence
$$
0 \rightarrow \sO_{\tilde X}(-D_l-mE') \rightarrow \sO_{\tilde X}(-mE') \rightarrow \sO_{\tilde X}(-mE') \otimes \sO_{D_l} \rightarrow 0.
$$
Pushing the sequence down to $X$ and using $\eqref{serre}$ we obtain an exact sequence
$$
0 \rightarrow \mu_* \sO_{\tilde X}(-D_l-mE') \rightarrow \mu_* \sO_{\tilde X}(-mE') \rightarrow (\mu|_{D_l})_* 
(\sO_{\tilde X}(-mE') \otimes \sO_{D_l}) \rightarrow 0.
$$
Since the inclusion $\sO_{\tilde X}(-D_l-mE') \rightarrow \sO_{\tilde X}(-mE')$ vanishes on $D_l$, its direct image 
$$
\mu_* \sO_{\tilde X}(-D_l-mE') \rightarrow \mu_* \sO_{\tilde X}(-mE')
$$
vanishes on $C_l = \mu(D_l)$. In other words we have 
$\sO_{\tilde X}(-D_l-mE')  \subset \sJ \cdot \sO_{\tilde X}(-mE')$ where
$\sJ$ is the full ideal sheaf of $C_l$. 
Thus we have an epimorphism
$$
\mu_* \sO_{\tilde X}(-mE') \otimes \sO_{C_l} \rightarrow (\mu|_{D_l})_* (\sO_{\tilde X}(-mE') \otimes \sO_{D_l}),
$$
which is generically an isomorphism. 
In particular, $\mu_* \sO_{\tilde X}(-mE') \otimes \sO_{C_l}$ is ample.
This proves the claim.

{\em Step 3. Conclusion.}
The ideal sheaf $\sK_{m}$ defines a 1-dimensional subspace $A \subset X$ such that $C \subset A$.
Let $\sI \subset \sO_X$ be the largest ideal sheaf on $X$ that coincides with $\sK_m$ in 
the generic point of every irreducible curve 
$C_l \subset C$\footnote{The sheaf $\sI$ defines the scheme-theoretic
image \cite[II,Ex.3.11(d)]{Har77} of the natural map $\cup_l C_{l, gen} \rightarrow X$ where we endowed
the generic points $C_{l, gen} \subset A$ with the open subscheme structure. 
}  such that $[C_l] \in R$.  
For every curve $C_l$, the natural map
$$
(\sK_m/\sK_m^2)|_{C_l} \rightarrow (\sI/\sI^2)|_{C_l}
$$
is generically an isomorphism. Since $(\sK_m/\sK_m^2)|_{C_l}$ is ample and $C_l$ is a curve, this implies
that $(\sI/\sI^2)|_{C_l}$ is ample. Thus $\sI/\sI^2$ is ample on its support $C$.
By \cite[Cor.3]{AV84} \cite{Gra62} there 
exists a holomorphic map $\nu: X \rightarrow X'$ contracting 
each connected component of $C$ onto a point. 
\end{proof} 

\subsubsection{Divisorial rays}

\begin{notation} {\rm \label{notationdivisorialray}
Under the Assumption \ref{assumption}, suppose that the extremal ray $R = \R^+ [\Gamma_{i_0}]$ is divisorial.
Since the divisor $K_X+\Delta$ is $\Q$-effective
and $(K_X+\Delta) \cdot R<0$ there exists an irreducible surface $S \subset X$ such that $S \cdot R < 0$.
In particular any curve $C \subset X$ with $[C] \in R$ is contained in $S$ and $S$ is covered by these curves.

Let $\nu: \tilde S \to S \subset X$ be the normalisation; 
then $\nu^*(\alpha|_S)$ is a nef class on $\tilde S$ and we may consider the nef reduction
$$ 
\tilde f: \tilde S \to \tilde T$$
with respect to $\nu^*(\alpha|_S)$ (cf. \cite[Thm.2.6]{8authors} and \cite[Thm.3.19]{HP15} for details on the K\"ahler case). 
Since $S$ is covered by curves that are $\alpha$-trivial, the surface $\tilde S$ is covered by curves 
that are $\nu^* (\alpha|_S)$-trivial.
By definition of the nef reduction
this implies
$$ n(\alpha) := n(\nu^* (\alpha|_S)) := \dim \tilde T \in \{0,1\}.$$}
\end{notation} 

\begin{lemma} \label{lemmacontractdivisorpoint}
Under the Assumption \ref{assumption}, 
suppose that the extremal ray $R$ is divisorial and $n(\alpha) = 0$. Then the surface 
$S$ can be blown down to a point $p$: there exists a bimeromorphic
morphism $\varphi: X \to Y$ 
to a normal compact  threefold $Y$ with $\dim \varphi(S) = 0$ such that $\varphi|_{X \setminus S}$ is an isomorphism onto $Y \setminus \{p\}$.
\end{lemma}

\begin{proof}
The proof is identical to the proof of \cite[Cor.7.7]{HP13a} which only uses that $S \cdot R<0$.
\end{proof}

The case  $n(\alpha)=1$ is much more subtle.
If $S \subset \supp \Delta$, we can suppose up to renumbering
that $S=\Delta_1$ and set $B=\sum_{i \geq 2} a_i \Delta_i$.
If $S \not\subset \supp \Delta$ we simply set $B=\Delta$.
In both cases we have by \eqref{differend} that
\begin{equation} \label{diff} 
K_{\tilde S}+B_{\tilde S} \sim_\Q \nu^* (K_X+S+B)|_S,
\end{equation} 
where $B_{\tilde S}$ is a canonically defined effective $\Q$-divisor. Let $\tilde C$ be a general fibre of the nef reduction $\tilde f: \tilde S \to \tilde T$.
Since $[\nu(\tilde C)] \in R$ we have
$$
(K_X+\Delta) \cdot \nu(\tilde C)<0 \qquad \mbox{and} \qquad S \cdot \nu(\tilde C)<0.
$$
Since $0 \leq a_1 \leq 1$ this implies
$$
(K_X+S+B) \cdot \nu(\tilde C) \leq (K_X+\Delta) \cdot \nu(\tilde C)<0.
$$
Thus we have $(K_{\tilde S}+B_{\tilde S}) \cdot \tilde C<0$,
in particular $\tilde C^2=0$ implies that
$\tilde C \simeq \PP^1$. The normal surface
$\tilde S$ being smooth in a neighbourhood of the fibre $\tilde C$, we
conclude $K_{\tilde S} \cdot \tilde C=-2$. Since
$B_{\tilde S} \cdot \tilde C \geq 0$ we arrive at

\begin{lemma} \label{lemmanone}
Under the Assumption \ref{assumption}, 
suppose that the extremal ray $R$ is divisorial and $n(\alpha) = 1$. Then the extremal ray $R$ is represented by 
a rational curve $C$ such that
$$
- (K_X+\Delta) \cdot C \leq 2.
$$
\end{lemma}

This estimate allows us to complete the proof of the cone theorem:

\begin{proof}[Proof of Theorem \ref{theoremNA}] \label{proofNA}
The only statement that is not part of Theorem \ref{theoremweakNA} is that in every $(K_X+\Delta)$-negative extremal ray $R_i$ 
we can find a rational curve $\Gamma_i$ such that $\Gamma_i \in R_i$ and
$$
0 < -(K_X+\Delta) \cdot \Gamma_i \leq 4.
$$ 
If the extremal ray $\Gamma_i$ is divisorial
with $n(\alpha)=1$, we conclude by Lemma \ref{lemmanone} (even if $(X,\Delta)$ is only lc).
If the extremal ray $\Gamma_i$ is small or divisorial
with $n(\alpha)=0$, the contraction exists by Theorem \ref{theoremsmallray} and Lemma \ref{lemmacontractdivisorpoint}. Since $(X,\Delta)$ is dlt, 
we may simply apply \cite{Kaw91} \cite[Thm.7.46]{Deb01}. 
\end{proof}

\begin{remark} {\rm In order to prove Theorem \ref{theoremNA} in the lc case, it remains to prove the existence of rational curves 
if the ray is divisorial contracting a divisor $S$ to a point or if the ray is small, with the bound $d = 4$ (cf. \cite{Fuj11b} for the projective case). The existence of rational curves in the first case just 
follows from the arguments preceding Lemma \ref{lemmanone}, which show that $S$ is uniruled. The existence of rational curves in the small case
requires some vanishing theorem which is not yet established in the K\"ahler case: if $W$ is the union of the lc centers, then $H^1(W,\sO_W) = 0.$ 
}
\end{remark} 

In order to prove the existence of the contraction in the case $n(\alpha)=1$ we would like to construct a fibration
$S \rightarrow T$ whose normalisation is the nef reduction $\tilde S \rightarrow \tilde T$.
At the moment we can only realise this strategy under an additional condition on the singularities of $S$.
We will use the notion of a semi-log-canonical (slc) surface; see Section 5. 

\begin{proposition} \label{propositionnone}
Under the Assumption \ref{assumption}, 
suppose that the extremal ray $R$ is divisorial and $n(\alpha) = 1$. Suppose moreover $S$ has slc singularities, i.e. $(S,0)$ is slc. 
Then there exists a morphism with connected fibres $S \rightarrow T$ onto a curve $T$ that contracts a curve $C \subset S$
if and only if $\alpha|_S \cdot C=0$. Moreover the contraction of the extremal ray $R$ exists.
\end{proposition}

\begin{proof}
Since $(S,0)$ is slc, it has normal crossing singularities in codimension one. 
The normalisation map $\nu: \tilde S \rightarrow S$ is finite, so the general curve 
$\nu(\tilde C)$ is contained in the locus where $S$
has normal crossing singularities. 
If $\nu(\tilde C)$ is in $S_{\nons}$ the proof
of \cite[Lemma 7.8]{HP13a} applies without changes, so
suppose that this is not the case. 
Let $$Z = \cup Z_k \subset S_{\sing}$$
be the union of curves $Z_k$
such that $\fibre{\nu}{Z_k}$ meets
the general curve $\tilde C$.
For each $k$ we set $\tilde Z_k := \nu^{-1}(Z_k)$ and set $\tilde Z := \cup \tilde Z_k$.
Since $S$ has normal crossings in the generic point of $Z_k$, the
natural map $\tilde Z_k \rightarrow Z_k$ has degree two.
Note now that every irreducible component of $\tilde Z$ is $\nu^* \alpha$-positive: otherwise $\tilde S$ would be connected by $\nu^* \alpha$-trivial curves, thus $n(\alpha)=0$ 
in contradiction to our assumption. In particular every curve in $\tilde Z$ dominates $\tilde T$. 

Since $\tilde Z$ maps into the normal crossings locus of $S$ 
we can decompose the differend (\ref{diff}) 
$$
B_{\tilde S} = \tilde Z + R
$$
where $R$ is an effective $\Q$-divisor with no common component with $\tilde Z$.
The intersection points of $\tilde Z$ and $\tilde C$
are contained in the smooth locus of $\tilde S$, so
$\tilde Z \cdot \tilde C$
is a positive integer. On the other hand 
we know that $K_{\tilde S} \cdot \tilde C = -2$
and
$$
-2 < 
(K_{\tilde S} + \tilde Z) \cdot \tilde C \leq
(K_{\tilde S} + \tilde Z + R) \cdot \tilde C < 0.
$$
Thus $\tilde Z \cdot \tilde C=1$. Since all the irreducible components of $\tilde Z$ surject onto $\tilde T$, this shows that
$\tilde Z$ is irreducible. In particular $Z$ itself
is irreducible.
Moreover we conclude that
\begin{enumerate}
\item the curves $Z$ and $C$ meet in a unique point $q$; and
\item we have $\fibre{\nu}{q}= \{ p_1, p_2 \}$ with $p_1 \in \tilde C$, but $p_2 \not\in \tilde C$.
\end{enumerate}
Since $\tilde C$ is general and $\nu$ is finite, the point
$p_2$ lies on another general fibre $\bar C$ such that
$\nu(\tilde C)$ and $\nu(\bar C)$ meet in $q$.
Let now $T$ be the unique irreducible component
of the cycle space $\Chow{S}$ such that the general point corresponds
to the cycle $\nu(\tilde C)+\nu(\bar C)$, and
let $\Gamma \to T'$ be the semi-normalisation \cite[p.156]{Kol07}  of the universal family
over $T$. By construction we have a natural bimeromorphic morphism
$\Gamma \rightarrow S$ and, by what precedes, this morphism
is an isomorphism in the neighbourhood of the general
fibre $F$ of $\Gamma \rightarrow T'$. 
Thus $F$ defines a Cartier divisor on $S$ such that $F^2=0$
and $\kappa(S, F)=1$. Thus some positive multiple of $F$ defines a morphism with connected fibres
$S \rightarrow T$ and one easily checks that this morphism realises the nef reduction with respect
to $\alpha|_S$. Now conclude as in \cite[Cor.7.9]{HP13a}.
\end{proof}


\subsection{Running the MMP} \label{subsectionrunMMP}

We first collect a number of results which we could not find in this form in the literature
and give an indication how to adapt the proof.

\begin{theorem} \label{theoremexistenceflip} (\cite{Sho92}, in the algebraic case \cite{Uta92})
Let $X$ be a normal $\Q$-factorial compact K\"ahler threefold. 
Let $\Delta$ be a  boundary divisor such that the pair $(X, \Delta)$ is dlt.
Let $\R^+ [\Gamma_i]$ be a $(K_X+\Delta)$-negative extremal ray $R$ in $\NAX$.
Suppose that the contraction of $\holom{\varphi}{X}{Y}$
is small. Then the flip $\holom{\varphi^+}{X^+}{Y}$ exists. Moreover $X^+$ is a normal $\Q$-factorial compact K\"ahler threefold and $(X^+, (\varphi^+)_*^{-1} \circ \varphi_* \Delta)$ is dlt. 
\end{theorem}

\begin{proof}
Since $X$ is $\Q$-factorial, the pair $(X, (1-\varepsilon) \Delta)$ is klt for every $0< \varepsilon < 1$.
Moreover for $0< \varepsilon \ll 1$ the divisor $K_X+(1-\varepsilon) \Delta$ is negative on the extremal ray $R$.
Thus the flip $\holom{\varphi}{X^+}{Y}$ exists by \cite[Thm]{Sho92}, see also \cite{Cor07} (and \cite[2.32]{Uta92} for passing to the limit). 
We alternatively may apply \cite{Uta92}, where the existence of log flips is reduced to the existence of terminal flips, \cite{Mor88}. 
The existence of terminal flips is now a local analytic construction. The reduction to the terminal case as given in sections 5,6 and 8 
in \cite{Uta92} works in the analytic setting as well\footnote{ 
It is important to notice that the construction is locally analytic near the contracted curves.}.
By Corollary \ref{corollarycontractionkaehler}, the complex space $Y$ is K\"ahler. 
Since the morphism $\varphi^+$ is projective, $X^+$ is again K\"ahler. 
For the remaining properties we refer to Proposition \ref{propositioncontraction}. 
\end{proof}

While the existence of flips is a local property \cite[6.7]{KM98}, this is not necessarily the case for termination results. However
in most cases we will only use special termination for dlt pairs which is much simpler to prove: 

\begin{theorem} \label{theoremspecialtermination} \cite{Sho92, Uta92}
Let $X$ be a normal $\Q$-factorial compact K\"ahler threefold. 
Let $\Delta$ be an effective reduced Weil divisor such that the pair $(X, \Delta)$ is dlt. Set $(X_0, \Delta_0) := (X, \Delta)$ and let 
$$
\left(
\merom{\varphi_i}{(X_i, \Delta_i)}{(X_{i+1}, \Delta_{i+1}:= (\varphi_i)_* \Delta_i}
\right)_{i \in I}
$$
be a sequence of $(K+\Delta)$-flips where $I \subset \N$. Then for all $i \gg 0$ the flipping locus is
disjoint from $\Delta_i$. 

In particular, if for every $i \in I$ the flipping locus of $\varphi_i$ is contained in $\Delta_i$, then the sequence of $(K+\Delta)$-flips terminates. 
\end{theorem}

\begin{proof}
The proof of \cite[Thm.4.2.1]{Fuj07} applies without changes: obviously the MMP exists for 
normal $\Q$-factorial compact K\"ahler {\em surfaces}, the discrepancy calculations 
are local properties, so they also apply in the non-algebraic setting.
\end{proof}

\begin{theorem} \label{theoremKtermination} \cite{Kaw92b, Uta92}
Let $X$ be a normal $\Q$-factorial compact  K\"ahler threefold such that $(X,0)$ is klt. 
Set $X_0:=X$ and let  
$$
\left(
\merom{\varphi_i}{X_i}{X_{i+1}}
\right)_{i \in I}
$$
be a sequence of $K$-flips. Then $I$ is finite, that is any sequence of $K$-flips terminates.
\end{theorem}

\begin{proof}
The proof of \cite[Thm.1]{Kaw92b} is based on three tools:
\begin{enumerate}
\item existence of terminal models for klt pairs \cite[Thm.5]{Kaw92b};
\item crepant extraction of a unique divisor with nonpositive discrepancy \cite[Lemma 6]{Kaw92b};
\item discrepancy calculations.
\end{enumerate}
The first two tools are available in the K\"ahler setting since we can argue as in Subsection \ref{subsectionbimeromorphic}.
Discrepancy calculations are local properties, so they also apply in the non-algebraic setting.
\end{proof}

As an application we can run the MMP for certain dlt pairs:

\begin{theorem} \label{theoremdltMMP}
Let $X$ be a normal $\Q$-factorial compact K\"ahler threefold that is not uniruled. 
Let $D \in |m K_X|$ be a pluricanonical divisor and set $B:=\supp D$. Suppose that the pair $(X, B)$ is dlt.
Then there exists a terminating $(K+B)$-MMP, that is, there exists a bimeromorphic map
$$
\varphi: (X, B) \dashrightarrow (X', B':=\varphi_*B)
$$
which is a composition of $K+B$-negative divisorial contractions and flips such that $X'$ is a 
normal $\Q$-factorial compact K\"ahler threefold, the pair $(X', B')$ is dlt and $K_{X'}+B'$ is nef.
\end{theorem}

\begin{proof}
Set $(X_0, B_0):=(X, B)$ and $D_0:=D$. 

{\em Step 1. Existence of the MMP.} If $K_{X_i}+B_i$ is nef, there is nothing to prove, so suppose that this is not the case.
Then there exists by Theorem \ref{theoremNA} a $K_{X_i}+B_i$-negative extremal ray $R_i \in \NA{X_i}$.
Since $X_i$ is not uniruled, the extremal ray is divisorial or small. 

If the extremal ray is small, the contraction $\holom{\psi_i}{X_i}{Y_i}$ exists by Theorem \ref{theoremsmallray}. By Theorem \ref{theoremexistenceflip}
the flip $\holom{\psi_i^+}{X_i^+}{Y_i}$ exists, and we denote by
$$
\merom{\varphi_i}{(X_i, B_i)}{(X_{i+1}:=X_i^+, B_{i+1}:=(\varphi_i)_* B_i)}
$$
the composition $(\psi_i^+)^{-1} \circ \psi_i$. By Theorem \ref{theoremexistenceflip} we know that 
$X_{i+1}$ is a normal $\Q$-factorial K\"ahler space and 
$(X_{i+1}, B_{i+1})$ is dlt, so we can continue the MMP.
Note finally that $D_{i+1}:=(\varphi_i)_* D_i$ is a pluricanonical divisor such that $\supp D_{i+1}=B_{i+1}$.

If the extremal ray is divisorial, let $S_i$ be the unique surface such that $S_i \cdot R_i<0$ (cf. Notation
\ref{notationdivisorialray}). Since  $K_{X_i}+B_i$ is $\Q$-effective and represented by
an effective $\Q$-divisor with support in $B_i$, the surface $S_i$ must be
an irreducible component of $B_i$. Since the pair $(X_i, B_i)$ is dlt, any of the irreducible components of $B_i$ is normal \cite[5.52]{KM98}. 
Moreover, there exists a boundary divisor $B_{i, S}$ such that
$$
K_{S_i}+B_{i, S} \sim_\Q K_{X_i}+B_i
$$
and the pair $(S_i, B_{i, S})$ is slc \cite[16.9.1]{Uta92}. Thus by Proposition \ref{propositionnone}
the contraction 
$$
\holom{\varphi_i}{(X_i, B_i)}{(X_{i+1},  B_{i+1}:=(\varphi_i)_* B_i)}
$$ 
exists and $X_{i+1}$ is a K\"ahler space by Corollary \ref{corollarycontractionkaehler}. Moreover by Proposition \ref{propositioncontraction}, 
$(X_{i+1},  B_{i+1})$ is dlt and $X_{i+1}$ is $\Q$-factorial, so we can continue the MMP.
Note finally that $D_{i+1}:=(\varphi_i)_* D_i$ is a pluricanonical divisor such that $\supp D_{i+1}=B_{i+1}$.

{\em Step 2. Termination.} 
For every $i \in I$ we have $\supp D_i=B_i$, hence $K_{X_i}+B_i$ is $\Q$-effective and represented by
an effective $\Q$-divisor with support in $B_i$. In particular the support of the extremal contraction is in $B_i$.
Thus any sequence of $(K+B)$-flips terminates by Theorem \ref{theoremspecialtermination}.
\end{proof}


 \section{Semi-log canonical K\"ahler surfaces} 

In this short section we gather the results concerning abundance in dimension 2 which will be used in the concluding sections. 
We refer to \cite[Ch.12]{Uta92} \cite[Ch.5]{Kol13} for basic definitions of the theory of semi-log canonical (slc) surfaces. 
All the statements in this section are shown for projective surfaces 
in \cite[Ch.12]{Uta92}, we will see that the proofs also apply to our case with some minor modifications.

Let $S$ be a reduced compact complex surface. We say that $S$ is in class $\mathcal C$ if the desingularisations of the irreducible components $S_i$
of $S$ are K\"ahler. The K\"ahler property is a bimeromorphic invariant of smooth compact surfaces, so this definition is independent of 
the choice of the desingularisations. 

\begin{proposition} \label{propositionslczero}
Let $S$ be an irreducible reduced compact complex surface in class $\mathcal C$. 
Let $\Delta$ be an effective $\mathbb Q$-Cartier divisor on $S$
such that the pair $(S,\Delta)$ is slc. Suppose that $K_S + \Delta$ is nef and numerically trivial.   
Then $K_S+\Delta$ is torsion, i.e. there exists a $m \in \N$ such that
$$
\sO_S(m(K_S+\Delta)) \simeq \sO_S.
$$
\end{proposition}

\begin{proof}
Let $\holom{\nu}{\tilde S}{S}$ be the normalisation, and let $\tilde \Delta$ be the differend (cf. Subsection \ref{subsectionadjunction}), so that we have
$$
K_{\tilde S}+\tilde \Delta \sim_\Q \nu^* (K_S+\Delta) \equiv 0,
$$
and the pair $(\tilde S, \tilde \Delta)$ is lc.

{\em 1st case. $\tilde \Delta \neq 0$.} In this case the surface $\tilde S$ is projective: 
indeed the anticanonical divisor $-K_S \equiv \Delta_S$
is pseudoeffective, so a $K$-MMP\footnote{The existence of a MMP for a compact K\"ahler surface with lc singularities can be established following our arguments in Section \ref{sectionMMP}.} $\tilde S \rightarrow \bar S$ terminates with a Mori fibre space $\bar S \rightarrow W$. 
Since $\bar S$ is a surface and $-K_{\bar S}$ is relatively ample we see that $\bar S$ is projective. Since $\tilde S \rightarrow \bar S$
is a projective morphism, the surface $\tilde S$ is projective. Thus the statements in \cite[Ch.12]{Uta92} apply.

{\em 2nd case. $\tilde \Delta=0$.} In this case $S$ is normal, so if $\holom{\pi}{\hat S}{S}$
is the minimal resolution, then we have
$$
\kappa(\hat S, K_{\hat S} + E) = \kappa(S, K_S),
$$
where $E$ is the canonical defined effective divisor such that $K_{\hat S}+E \sim_\Q \pi^* K_S \equiv 0$.
Now if $E \neq 0$, then the smooth surface $\hat S$ is uniruled, hence projective, and therefore $\kappa(\hat S, K_{\hat S} + E)=0$ by \cite[Thm.12.1.1]{Uta92}.
If $E = 0$, then $S$ is a compact K\"ahler surface with numerically trivial canonical bundle, so $\kappa(\hat S, K_{\hat S})=0$ by the classification of compact
surfaces \cite{BHPV04}.
\end{proof}

\begin{lemma} \label{lemmaslcone}
Let $S$ be a normal compact K\"ahler surface, and let $\Delta$ be a boundary divisor on $S$.
Suppose that the pair $(S, \Delta)$ is lc, that the class $K_S+\Delta$ is nef and that $\nu(K_S+\Delta)=1$.
Then $K_S + \Delta$ is semi-ample. 
\end{lemma}

\begin{proof}
Let $\holom{\mu}{S'}{S}$ be a dlt model,  cp. Theorem \ref{theoremdltmodel}, then $(S',\Delta')$ is dlt and $K_{S'}+\Delta' \sim_\Q \mu^* (K_S+\Delta)$. In particular
$K_S+\Delta$ is semiample if and only if $K_{S'}+\Delta'$ is semiample. Thus we can assume without loss of generality
that $(S, \Delta)$ is dlt.

Dlt singularities are rational \cite[Thm.5.22]{KM98}, so the surface $S$ is projective if and only if some desingularisation
is projective \cite{Nam02}. Thus if the algebraic dimension $a(S)$ is two,
we conclude by \cite[Thm.12.1.1]{Uta92}. If $a(S)=1$, let $\holom{f}{S}{C}$ be the algebraic reduction.
The general fibre $F$ is an elliptic curve, so $K_S|_F \simeq \sO_F$. Moreover {\em every} curve in $S$ is contracted
by $f$, so $\Delta \cap F = \emptyset$. Since $K_S+\Delta$ is nef, this implies 
$$
K_S+\Delta \sim_\Q f^* A,
$$
with $A$ an ample $\Q$-divisor on $C$. Thus $K_S+\Delta$ is semiample. Finally we want to exclude the case $a(S)=0$:
note first that $S$ contains only finitely many curves, so a desingularisation $\hat S \rightarrow S$ is not uniruled. Thus
we have
$$
\kappa(K_S) \geq \kappa(K_{\hat S}) \geq 0,
$$
in particular $K_S+\Delta$ is $\Q$-linearly equivalent to an effective divisor. Yet the intersection form on a surface with $a(S)=0$
is negative definite, hence $(K_S+\Delta)^2<0$. In particular $K_S+\Delta$ is not nef, a contradiction.
\end{proof}

\begin{proposition} \label{propositionslcone}
Let $S$ be a connected reduced compact complex surface in class $\mathcal C$. Let $\Delta$ be an effective $\mathbb Q$-Cartier divisor on $S$
such that the pair $(S,\Delta)$ is slc. Suppose that $K_S + \Delta$ is nef with numerical dimension $\nu(K_S + \Delta) \leq 1$.  
Suppose also that for every irreducible component $T \subset S$ the restriction $(K_S+\Delta)|_T$ is not zero.

Then the divisor $K_S + \Delta $ is semi-ample. 
\end{proposition}

\begin{proof}
Let $\holom{\nu}{T}{S}$ be the composition of a minimal semi-resolution (cf. \cite[Defn.12.2.1, Prop.12.2.3]{Uta92})
with the normalisation. Since $(S, \Delta)$ is slc there exists for every irreducible component $T_i \subset T$ a boundary divisor
$\Delta_i$ such that
$$
K_{T_i} + \Delta_i \sim_\Q \nu^* (K_S+\Delta).
$$
By our assumption $K_{T_i} + \Delta_i$ is nef of numerical dimension one, so $K_{T_i} + \Delta_i$ is semiample
by Lemma \ref{lemmaslcone}. The whole point is now to prove that sufficiently many pluricanonical sections descend to $S$.
However the proof of this descent theorem in \cite[Ch.12.4]{Uta92} is based on \cite[Prop.12.3.2, Thm.12.3.4]{Uta92} which
are statements about proper morphisms. Thus they apply in our setting.
\end{proof} 

\begin{theorem} \label{theoremslcsurjective}
Let $S$ be a compact connected reduced complex space of dimension two in class $\mathcal C$ with slc singularities.
Then the natural maps induced by $\C \subset \sO_S$ 
$$
H^p(S, \C) \rightarrow H^p(S, \sO_S)
$$ 
are surjective for every $p \in \N$. 
\end{theorem}

\begin{proof}  
Simply note that the proof of Theorem 12.1.2 of \cite{Uta92} works practically word by word. The main ingredient \cite[12.2.8]{Uta92} is  about 
germs of slc surfaces and works in our situation. 
\end{proof}


\section{Abundance: reduction steps} \label{sectionreduction}

This section corresponds to the preparatory lemmas in \cite[Ch.13,Ch.14]{Uta92}. While the statements and the basic strategy of proof
are quite similar we have to be more careful since we do not have a full log-MMP at our disposal. Moreover we make several
points more precise which allows us to conclude quicker� in Section \ref{sectionabundance}.

\begin{lemma} \label{lemmaMMP1}
Let $X$ be a normal $\Q$-factorial compact K\"ahler threefold with terminal singularities. 
Suppose that $K_X$ is nef and $\nu(X)>0$. 
Then there exists a normal $\Q$-factorial compact K\"ahler threefold $X'$ that is bimeromorphic
to $X$ and a $D' \in |m K_{X'}|$ with the following properties:
\begin{enumerate}
\item Set $B':=\supp D'$. Then the pair $(X', B')$ is dlt
and $X' \setminus B'$ has terminal singularities.
\item The divisor $K_{X'}+B'$ is nef and we have
\begin{enumerate}
\item $\nu(X) = \nu(K_{X'}+B')$; and
\item $\kappa(X) = \kappa(K_{X'}+B')$.
\end{enumerate}
\end{enumerate}
\end{lemma}

The proof needs some technical preparation:

\begin{remark} \label{remarkkodairadimension}
Let $X$ be a normal $\Q$-factorial compact  K\"ahler space,
and let $D \in |mK_X|$ be an effective pluricanonical divisor.
If  we set $B=\supp D$, then
$$
\kappa(K_X) = \kappa(K_X+B).
$$
Indeed, $B$ being effective, the inequality $\kappa(K_X) \leq  \kappa(K_X+B)$ is obvious. Yet we also have an inclusion
of effective divisors $B \subset D$, so we get 
$$
\kappa(K_X+B) \leq \kappa(K_X+D)=\kappa((m+1)K_X) = \kappa(X).
$$
\end{remark}

The following basic lemma has been shown in the algebraic case in \cite[13.2.4]{Uta92}. 
\begin{lemma} \label{lemmaimagedivisorial}
Let $\holom{\psi}{V}{W}$ be a bimeromorphic  proper K\"ahler morphism between normal complex spaces. Let $D$ be an effective
$\Q$-Cartier $\Q$-divisor that is $\psi$-nef. Then we have
$$
\supp \psi_* (D) = \psi(\supp D),
$$
i.e., the image of $D$ has pure codimension one.
\end{lemma}

\begin{proof}
The inclusion $\supp \psi_* (D) \subset \psi(\supp D)$ is trivial. Fix now a point $w \in \psi(\supp D) \subset W$. 
Arguing by contradiction we assume that all the irreducible components
of $\psi(\supp D)$ passing through $w$ have dimension at most $\dim W-2$.
Thus, up to replacing $W$ by an analytic neighbourhood of $w$, we assume that $\psi(\supp D)$ has dimension at most $\dim W-2$. 
Then $\psi_*(-D)=0$ and $-(-D)$ is $\psi$-nef,
contradicting the negativity lemma, see e.g. \cite[3.6.2]{BCHM10}, unless $D = 0.$ 
\end{proof}

The following crucial lemma seems to be well-known to experts and is used in several places without further mentioning,  but we prefer to write it down in detail:

\begin{lemma} \label{lemmaisomorphism}
Let $X=:X_0$ be a normal $\Q$-factorial compact K\"ahler space, and let $M=:M_0$ be an effective Cartier divisor on $X$ that is nef.
Let $N=:N_0$ be an effective Cartier divisor such that $\supp M=\supp N$, and let
$$
\left(
\varphi_i: X_i \dashrightarrow X_{i+1}
\right)_{i=0, \ldots, n}
$$
be a finite sequence of $N$-negative contractions, that is every $X_i$ is a normal $\Q$-factorial compact K\"ahler space and $\varphi_i$
is the divisorial contraction or flip of an extremal ray $R_i \in \NA{X_i}$ that is $N_i$-negative\footnote{As usual
we define inductively $N_{i+1}:= (\varphi_i)_* N_i$. Cp. \cite[2.26]{Uta92}}.
Then for every $i=0, \ldots, n$ the $N$-MMP induces an isomorphism 
$$
(X_0 \setminus \supp M_0) \simeq (X_{i+1} \setminus \supp M_{i+1}).
$$
\end{lemma}

\begin{remark*} The essence of Lemma \ref{lemmaisomorphism}  is the following. Given the flip $\varphi: X_i \dasharrow X_{i+1} $ with contractions
$f_i: X_i \to Y_i$ and $f_i^+: X_{i+1} \to Y_i,$ then the exceptional locus of $f_i^+$ is contained in $M_{i+1}.$ 
\end{remark*} 

\begin{proof}
For every $i=1, \ldots, n$ we set $M_{i+1}:= (\varphi_i)_* M_i$. Moreover we denote by $\Gamma_i$
the normalisation of the graph of $\varphi_i$ and by $\holom{p_i}{\Gamma_i}{X_i}$
and $\holom{q_i}{\Gamma_i}{X_{i+1}}$ the natural maps.

We will construct inductively a sequence of normal compact K\"ahler spaces $V_i$ admitting bimeromorphic 
morphisms $\holom{f_i}{V_i}{X}$ and $\holom{g_i}{V_i}{X_{i+1}}$ such that 
\begin{enumerate}
\item the support of $f_i^* M$ contains the $g_i$-exceptional locus;
\item $g_i$ factors through $\Gamma_i$, that is there exists a bimeromorphic map $\holom{\beta_i}{V_i}{\Gamma_i}$
such that $g_i = q_i \circ \beta_i$.
\end{enumerate}
Assuming this for the time being, let us see how to conclude: arguing by induction we suppose that we have an isomorphism
$$(X_0 \setminus \supp M_0) \simeq (X_{i} \setminus \supp M_{i}).$$
Since $\supp M_i = \supp N_i$ and $\varphi_i$ is $N_i$-negative, the image of the $p_i$-exceptional locus
in contained in $M_i$.
Thus we are done if we show that the image of the $q_i$-exceptional 
locus is contained in $M_{i+1}$. Since $X_0 \dashrightarrow X_{i+1}$ does not extract a divisor we have
$$
M_{i+1} = (g_i)_* f_i^* M.
$$
Yet by the properties above $f_i^* M$ contains the strict transforms of the $q_i$-exceptional divisors\footnote{Note that the $q_i$-exceptional locus is divisorial since $X_{i+1}$ is $\Q$-factorial.}. 
Thus $g_i(\supp f_i^* M)$ contains the image of the $q_i$-exceptional locus. By Lemma \ref{lemmaimagedivisorial} this implies that $\supp (g_i)_* f_i^* M$ contains the image of the $q_i$-exceptional locus.

{\em Proof of the claim.} For the start of the induction we simply set $V_0:=\Gamma_0$. Indeed since 
$\supp M_0=\supp N_0$ contains the exceptional locus of the extremal contraction, the divisor $p_0^* M_0$ contains
the $q_0$-exceptional locus.

For the induction step we make a case distinction: if $\varphi_i$ is divisorial we simply set $V_i:=V_{i-1}$ and note that
$M_i$ contains the exceptional divisor since $\supp M_i = \supp N_i$.
If $\varphi_i$ is a flip we define $V_i$ as the normalisation of the graph of the bimeromorphic map $V_{i-1} \dashrightarrow \Gamma_i$,
and denote by $\holom{\alpha_i}{V_i}{V_{i-1}}$ and $\holom{\beta_i}{V_i}{\Gamma_i}$ the natural maps.
We set $f_i:=f_{i-1} \circ \alpha_i$ and $g_i := q_i \circ \beta_i$ and summarise the situation in a commutative
diagram:
$$
\xymatrix{
& & V_i  \ar[ld]_{\alpha_i}  \ar[rd]^{\beta_i} & &
\\
& V_{i-1}  \ar[ld]_{f_{i-1}}  \ar[rd]^{g_{i-1}} & & \Gamma_i \ar[ld]_{p_i} \ar[rd]^{q_i} &
\\
X & & X_{i} &  & X_{i+1} 
}
$$
All we have to show is that  
the support of $f_i^* M$ contains the $g_i$-exceptional locus. 
Note that by our induction hypothesis $f_{i-1}^* M$ contains the $g_{i-1}$-exceptional locus. Moreover we have
$$
M_i = (g_{i-1})_* f_{i-1}^* M
$$ 
and $\supp M_i=\supp N_i$. In particular $(g_{i-1})_* f_{i-1}^* M$ contains the image of the $p_i$-exceptional locus (which is of
course equal to the $q_i$-exceptional locus). Using these two properties the claim follows by elementary set-theoretic computations.
\end{proof}

\begin{proof}[Proof of Lemma \ref{lemmaMMP1}] 
Since $\kappa(X) \geq 0$ by \cite{DP03}, 
there exists a $m \in \N$ and an effective Cartier divisor 
$D \in |mK_X|$. Since $\nu(X)>0$, the divisor $D$ is not zero. 
Let $$\holom{\mu}{X_0}{X}$$ be a log-resolution of the 
divisor $D$ in the following sense: the map $\mu$ is an isomorphism on $X \setminus D$,
the support of $\mu^* D$ has simple normal crossings and
$X_0$ is smooth in a neighbourhood of $\mu^* D$\footnote{This last
condition can be satisfied since $X$ is a terminal threefold, so it has 
only finitely many singular points.}. Thus if we set $B_0:=(\mu^* D)_{\red}$, the pair
$(X_0, B_0)$
is dlt and $X_0 \setminus B_0$ has terminal singularities. 
Note also that since $X$ has terminal
singularities, we have
$$
K_{X_0} \sim_\Q \mu^* K_X + E
$$
with $E$ an effective $\mu$-exceptional $\Q$-divisor. By construction 
$$
\supp E \subset B_0,
$$
so the pluricanonical divisor 
$$
D_0 := \mu^* D+mE \in |m K_{X_0}|
$$
satisfies $\supp D_0=B_0$. By Remark \ref{remarkkodairadimension}
we have
$$\kappa(X) = \kappa(X_0) = \kappa(K_{X_0}+B_0).$$
By Theorem \ref{theoremdltMMP} there exists a terminating $(K+B)$-MMP
$$
\left(
\varphi_i: (X_i, B_i) \dashrightarrow (X_{i+1}, B_{i+1})
\right)_{i=0, \ldots, n}
$$
where $\varphi_i$ is the flip or divisorial contraction
of a ray $R_i$. Then for every $i \in \{1, \ldots, n\}$
$$
D_{i+1} := \varphi_* D_i  \in |m K_{X_{i+1}}|
$$
is a pluricanonical divisor such that $\supp D_{i+1}=B_{i+1}$. 

For every $0 \leq i \leq n$, if $C \subset X_i$ is a curve with class  $[C] \in R_i$, then 
$$(K_{X_i} + B_i) \cdot C = ({{1}�\over {m}}D_i + B_i) \cdot C < 0,$$
and since $\supp D_i = B_i$, the curve $C$ is contained in $B_i.$ Thus 
the exceptional locus of the contraction of the extremal ray
$R_i$ is contained in $B_i$. Now apply Lemma \ref{lemmaisomorphism} to establish an isomorphism
$$
\varphi_i|_{X_i \setminus B_i}: (X_i \setminus B_i) \rightarrow (X_{i+1} \setminus B_{i+1}).
$$
In conclusion we obtain that $X_{i+1} \setminus B_{i+1}$ has terminal singularities. 

The Kodaira dimension of $K+B$ being invariant under a $(K+B)$-MMP, so Remark \ref{remarkkodairadimension} yields
$$
\kappa(X) = \kappa(K_{X_0}+B_0) =  \kappa(K_{X_{n+1}}+B_{n+1}).
$$
Let now $\Gamma$ be a desingularisation of the graph of the bimeromorphic map $X_0 \dashrightarrow X_{n+1}$, and denote by
$\holom{p}{\Gamma}{X_0}$ and $\holom{q}{\Gamma}{X_{n+1}}$ the natural projections.
Then the effective divisors 
$$
(\mu \circ p)^* D \simeq (\mu \circ p)^* mK_X
$$
and 
$$
q^* (D_{n+1}+mB_{n+1}) \simeq m q^*(K_{X_{n+1}}+B_{n+1})
$$
are both nef and have the same support. As in \cite[11.3.3]{Uta92}, arguing with a K\"ahler  class, this implies
$$
\nu(X) = \nu(K_{X_{n+1}}+B_{n+1}).
$$
\end{proof}

\begin{lemma} \label{lemmaconnected}
Let $X$ be a normal compact K\"ahler space of dimension $n \geq 2$, and let
$L$ be a nef $\Q$-Cartier $\Q$-divisor. Let $S \subset X$ be an effective $\Q$-Cartier $\Q$-divisor such that $L-S$ is nef.
Let $T \subset X$ be a prime divisor such that $L|_T \equiv 0$ and $T \not\subset \supp S$. Then 
$$
T \cap \supp S = \emptyset.
$$
\end{lemma}

\begin{proof}
Indeed,  if $T \cap \supp S \neq \emptyset$, then $-S|_T$ is a non-zero antieffective divisor. Thus the restriction
$
(L-S)|_T \equiv -S|_T
$
is not nef, a contradiction.
\end{proof}

\begin{lemma} \label{lemmaprepareMMP}
Let $X$ be a normal $\Q$-factorial compact K\"ahler threefold.
Suppose that there exists a $D \in |m K_{X}|$ with the following properties:
\begin{itemize}
\item Set $B:=\supp D$. The pair $(X, B)$ is dlt and $X \setminus B$ has terminal singularities.
\item The divisor $K_{X}+B$ is nef. 
\end{itemize}
Let $S \subset B$ be a non-zero effective Weil divisor such that the following holds: for every irreducible component $T \subset B-S$ we have
$$
(K_X+B)|_{T} \equiv 0.
$$
Set $X_0 := X, D_0 := D$ and $B_0 := B$. 
Then there exists a finite sequence $K+B-S$-negative contractions
$$
\left(
\varphi_i: (X_i, B_i-S_i) \dashrightarrow (X_{i+1}, B_{i+1}-S_{i+1})
\right)_{i=0, \ldots, n}
$$
of a ray $R_i \in \NA{X_i}$ such that the following properties hold:
\begin{enumerate}
\item The pair $(X_{i+1}, B_{i+1}-S_{i+1})$ is dlt.
\item Set $D_{i+1} := (\varphi_i)_* D_i$.
Then the divisor $D_{i+1} \in |m K_{X_{i+1}}|$ satisfies $\supp D_{i+1}=B_{i+1}$.
\item We have $(K_{X_i}+B_i) \cdot R_i = 0$ and the divisor $K_{X_{i+1}}+B_{i+1}$ is nef.
For every irreducible component $T_{i+1} \subset B_{i+1}-S_{i+1}$
we have
$$
(K_{X_{i+1}}+B_{i+1})|_{T_{i+1}} \equiv 0.
$$
\item $X_{i+1}$ is $\bQ-$factorial;  the pair $(X_{i+1}, B_{i+1})$ is lc and  $X_{i+1} \setminus B_{i+1}$  has terminal singularities.
\item We have $(B_{n+1}-S_{n+1}) \cap S_{n+1} = \emptyset$ and $S_{n+1} \neq 0$.
\end{enumerate}
\end{lemma}

\begin{remark*}
Note that we do not claim that $K_{X_{n+1}}+B_{n+1}-S_{n+1}$ is nef, in general this will not be true.
\end{remark*}

\begin{proof}
If $(B-S) \cap S = \emptyset$, there is nothing to prove, so suppose $(B-S) \cap S \ne \emptyset.$ 
Thus $K_X + B - S$ is not nef by Lemma \ref{lemmaconnected}. Note that 
$(X_0,B_0-S_0)$ is dlt.

{\em Step 1. Existence of the MMP.} We proceed by induction and assume that $X_i, B_i $ and $D_i$  are already constructed such that properties
$a), b)$  and $d)$ hold at level $i$, moreover $K_{X_i} + B_i$ is nef and $K_{X_i} + B_i \vert_{T_i} \equiv 0 $ for all irreducible components $T_i$ of $B_i - S_i.$ 
If there is no such $T_i$ meeting $S_i$, then we stop and set $n+1 = i.$ Thus we may assume that 
$T_i \cap S_i \neq \emptyset$ for some component $T_i.$ We first have to find the extremal ray $R_i.$ 
Notice that the restriction
$$
(K_{X_i}+B_i-S_i)|_{T_i} \equiv -S_i|_{T_i} 
$$
is not nef, so by Corollary \ref{corollaryalgebraicallynef} there exists a curve $Z_i \subset T_i \subset (B_i-S_i)$ such that
$$
(K_{X_i}+B_i-S_i) \cdot Z_i < 0.
$$
The inclusion $Z_i \subset (B_i-S_i)$ yields$(K_{X_i}+B_i) \cdot Z_i = 0$ by $c)$. 
By the cone theorem \ref{theoremNA} applied to the pair $(X_i, B_i-S_i)$, there exists a decomposition
\begin{equation} \label{decomposezii}
Z_i = \sum \lambda_j C_{i,j} + M_i,
\end{equation}
where $\lambda_j>0$, the $C_{i,j}$ are irreducible curves generating a $(K_{X_i}+B_i-S_i)$-negative extremal 
in $\NA{X_i}$ and $M_i \in \NA{X_i}$ such that 
$$(K_{X_i}+B_i-S_i) \cdot M_i \geq 0.$$
Let $R_i$ be the extremal ray generated by the curve $C_{i,1}$. Since $K_{X_i}+B_i$ is nef and
$(K_{X_i}+B_i) \cdot Z_i = 0$, the decomposition \eqref{decomposezii} implies that $(K_{X_i}+B_i) \cdot R_i = 0$.

Since $\supp D_i = B_i$, the $\Q$-Cartier divisor 
$K_{X_i}+B_i-S_i$ is $\Q$-linearly equivalent to an effective divisor with support in $B_i$. In particular
every curve $C \subset X_i$ such that $[C] \in R_i$ is contained in $B_i$. Since the pair $(X_i, B_i)$
is lc this shows that the support of the extremal ray is contained in an irreducible surface with slc singularities, \cite[16.9]{Uta92}. 
Thus the contraction of $R_i$ exists by Theorem \ref{theoremcontraction}. Moreover, if the ray is small, then the flip
exists by Theorem \ref{theoremexistenceflip}. We denote by
$$
\varphi_i: (X_i, B_i-S_i) \dashrightarrow (X_{i+1}, B_{i+1}-S_{i+1})
$$
the flip or divisorial contraction of $R_i$. As in the proof of Lemma \ref{lemmaMMP1} we see that the properties $a)$ and $b)$ 
hold at level $i+1$.

Since $(K_{X_i}+B_i) \cdot R_i=0$, Lemma \ref{lemmacrepantMMP} proves
that $K_{X_{i+1}}+B_{i+1}$ is nef,
that the pair $(X_{i+1}, B_{i+1})$ is lc. 
and that for every irreducible component 
component $T_{i+1} \subset B_{i+1}-S_{i+1},$ the restriction $(K_{X_{i+1}}+B_{i+1})|_{T} \equiv 0$ .
Finally, the divisor $K_{X_i}+B_i$ being nef and $\Q$-linearly equivalent to an effective divisor with support in $B_i$,
Lemma \ref{lemmaisomorphism} implies that 
$(X_{i+1} \setminus B_{i+1}) \simeq (X_i \setminus B_i)$
has terminal singularities.

{\em Step 2. Termination of the MMP.} By Theorem \ref{theoremspecialtermination} we know that after finitely many steps the 
flipping locus is disjoint from $B_i-S_i$. Yet a $K+B-S$-negative extremal contraction that is disjoint from $B-S$ is also a $K$-negative
extremal contraction. Thus the sequence terminates by Theorem \ref{theoremKtermination}. 
\end{proof}

\begin{lemma} \label{lemmanuoneMMP}
In the situation of Lemma \ref{lemmaprepareMMP} suppose additionally that $\nu(K_X+B)=1$,
and let $S \subset B$ be an irreducible component. 

Then there exists a normal $\Q$-factorial compact K\"ahler threefold $X'$ that is bimeromorphic to $X$ and has
a $D' \in |m K_{X'}|$ with the following properties:
\begin{enumerate}
\item Set $B':=\supp D'$. The pair $(X', B')$ is lc and $X' \setminus B'$ has terminal singularities.
\item The divisor $K_{X'}+B'$ is nef with $\nu(K_{X'}+B')=1$. Moreover,
$
\kappa(K_{X}+B) = \kappa(K_{X'}+B').
$
\item There exists an irreducible component $S' \subset B'$ that is a connected component of $B'$.
\item The pair $(X', B'-S')$ is dlt.
\end{enumerate}
\end{lemma}

\begin{proof}
We first check that $B-S$ satisfies the positivity condition in Lemma \ref{lemmaprepareMMP}: indeed if $T$ is any irreducible component 
of $\supp (D),$ then $(K_X + B)|_T \equiv 0$, otherwise $\nu(K_X + B) \geq 2.$ 

Using the notation of Lemma \ref{lemmaprepareMMP}, set $X' = X_{n+1}$ etc.  We know that $S_{n+1}$ is a non-zero irreducible divisor that is 
disjoint from $B_{n+1}-S_{n+1}$. Thus the property $c)$ holds. 
Since all the contractions in Lemma \ref{lemmaprepareMMP} are $(K+B)$-trivial, the Kodaira dimension is invariant: 
$$
\kappa(K_{X}+B) = \kappa(K_{X_{n+1}}+B_{n+1}).
$$
The other properties were already shown in Lemma \ref{lemmaprepareMMP}.
\end{proof}

\begin{lemma} \label{lemmaconnected2}
Let $X$ be a normal compact K\"ahler space with rational singularities of dimension $n \geq 2$. Let $D$ be an effective Cartier divisor.
Suppose that $D$ is nef and $\nu(D) \geq 2$. Then the support of $D$ is connected.
\end{lemma}

\begin{proof}
It is sufficient to prove that $H^1(X, \sO_X(-D)) = 0$. Since $X$ has rational singularities we may replace it by a resolution of singularities
without changing this cohomology group. Then we have
$$
H^1(X, \sO_X(-D))  \simeq H^{n-1}(X, K_X \otimes \sO_X(D)) = 0
$$
by \cite[Thm.0.1]{DP03}.
\end{proof}

\begin{lemma} \label{lemmanutwoMMP}
In the situation of Lemma \ref{lemmaprepareMMP} suppose additionally that $\nu(K_X+B)=2$.

Then there exists a normal $\Q$-factorial compact K\"ahler threefold $X'$, bimeromorphic to $X$ such that $(X',0)$ is klt 
and a divisor $D' \in |m K_{X'}|$ with the following properties:
\begin{enumerate}
\item Set $B':=\supp D'$. The pair $(X', B')$ is lc and $X' \setminus B'$ has terminal singularities.
\item The divisor $K_{X'}+B'$ is nef with $\nu(K_{X'}+B')=2$. Moreover,
$
\kappa(K_{X}+B) = \kappa(K_{X'}+B').
$
\item For every irreducible component $T' \subset B'$ we have $(K_{X'}+B')|_{T'} \not \equiv 0$.
\item $(K_{X'}+B') \cdot K_{X'}^2 \geq 0$.
\end{enumerate}
\end{lemma}

\begin{proof}
The construction of the K\"ahler space $X'$ proceeds in two steps. 

{\em Part I. Eliminating the $(K+B)$-trivial components.}
Let $S \subset B$ be the union of all the irreducible components $T \subset B$ such that $(K_X+B)|_{T} \not\equiv 0$.
Since $\nu(K_X+B)=2$, the set $S$ is not empty. Thus we can apply Lemma \ref{lemmaprepareMMP}
to obtain a sequence of $(K+B-S)$-negative contractions
$$
\psi: X \dashrightarrow \hat X
$$
where $\hat X$ is a normal $\Q$-factorial compact K\"ahler threefold carrying 
a divisor $\hat D \in |m K_{\hat X}|$ subject to the following properties:
\begin{enumerate}
\item[$(\alpha)$] Set $\hat B:=\supp \hat D$. The pair $(\hat X, \hat B)$ is lc and $\hat X \setminus \hat B$ has terminal singularities.
\item[$(\beta)$] The divisor $K_{\hat X}+\hat B$ is nef with $\nu(K_{\hat X}+\hat B)=2$. Moreover,
$\kappa(K_{X}+B) = \kappa(K_{\hat X}+\hat B).$
\item[$(\gamma$)] Set $\hat S := \psi_* S$.  Then the pair $(\hat X, \hat B- \hat S)$ is dlt.
\item[$(\delta)$] ${\rm supp}(\hat B - \hat S) \cap \hat S = \emptyset$ and $\hat S \neq 0$.
\end{enumerate}
Notice also that $K_{\hat X}+\hat B$ is $\Q$-linearly equivalent to an effective
divisor with support in $\hat B$. By Lemma  \ref{lemmaconnected2} this implies that
$\hat B = \hat S + (\hat B - \hat S)$ is connected. Thus property $(\delta)$ implies that 
$\hat B - \hat S=0$. Hence $(K_{\hat X}+\hat B)|_{T}$
is non-zero for every irreducible component $T \subset \hat B$.

{\em Part II. Eliminating $K$-negative curves.}
Setting 
$$
X_0 :=\hat X, \ B_0 := \hat B, \ D_0 := \hat D,
$$
we will next construct  a finite sequence of $K$-negative contractions and flips 
$$
\left(
\varphi_i: X_i  \dashrightarrow X_{i+1}
\right)_{i=0, \ldots, n}
$$
of extremal rays $R_i \in \NA{X_i}$ such that the following properties hold:
\begin{enumerate}
\item[(1)]  $X_{i+1} $ is $\bQ-$factorial and $(X_{i+1},0)$ is klt.
\item[(2)] Set $D_{i+1} := (\varphi_i)_* D_i$.
Then $D_{i+1} \in |m K_{X_{i+1}}|$ satisfies $\supp D_{i+1}=B_{i+1}$.
\item[(3)] We have $(K_{X_i}+B_i) \cdot R_i = 0$, the divisor $K_{X_{i+1}}+B_{i+1}$ is nef and $\nu(K_{X_{i+1}}+B_{i+1})=2$.
\item[(4)] The pair $(X_{i+1}, B_{i+1})$ is lc and  $X_{i+1} \setminus B_{i+1}$ has terminal singularities.
\item[(5)] For every irreducible component $T \subset B_{i+1} $ we have $(K_{X_{i+1} }+B_{i+1})|_{T} \neq 0$.
\item[(6)] $(K_{X_{n+1}}+B_{n+1}) \cdot K_{X_{n+1}}^2 \geq 0$.
\end{enumerate}

Note first that $(\hat X,0)$ is klt singularities since $(\hat X, \hat B - \hat S)=(\hat X, 0)$ is dlt \cite[Prop.2.41]{KM98}.
If $(K_{\hat X}+ \hat B) \cdot K_{\hat X}^2 \geq 0$, there is nothing to prove, thus we may assume that 
$$(K_{\hat X}+ \hat B) \cdot K_{\hat X}^2 < 0.$$

We therefore start by showing that $(K_{X_i}+B_i) \cdot K_{X_i}^2<0$ implies the existence of a $K_{X_i}$-negative 
contraction that is $(K_{X_i}+B_i)$-trivial.
Let $B_i = \sum B_{i,l}$
be the decomposition of $B_i$ in its irreducible components. Since the pair $(X_i, B_i)$ is lc, the surfaces $B_{i,l}$ have
slc singularities. Moreover, by adjunction \cite[Prop.16.9]{Uta92}, there exists a boundary divisor $\Delta_{i,l}$ such that
$$
K_{B_{i,l}}+\Delta_{i,l} \sim_\Q (K_{X_i}+B_i)|_{B_{i,l}},
$$
and the pair $(B_{i,l}, \Delta_{i,l})$ is slc.
By Proposition \ref{propositionslcone} this implies that 
$$
(K_{X_i}+B_i)|_{B_{i,l}} \sim_\Q Z_{i,l}
$$ 
with $Z_{i,l}$ an effective $1$-cycle.

Suppose that $K_{X_i}|_{B_{i,l}} \cdot Z_{i,l} \geq 0$
for all $l$. Since 
$$K_{X_i} \sim_\Q  \sum a_{i,l} B_{i,l}$$ 
with $a_{i,l}>0$ we conclude
$$
(K_{X_i}+B_i) \cdot K_{X_i}^2 = \sum  a_{i,l}  (K_{X_i}+B_i) \cdot K_{X_i} \cdot B_{i,l} = 
\sum  a_{i,l}  K_{X_i}|_{B_{i,l}} \cdot Z_{i,l} \geq 0,
$$ 
a contradiction. Thus we can suppose (up to renumbering) that
$$K_{X_i}|_{B_{i,1}} \cdot Z_{i,1} < 0.$$

By the cone theorem \ref{theoremNA} applied to the pair $(X_i, 0)$, there exists a decomposition
\begin{equation} \label{decomposezi}
Z_{i,1} = \sum \lambda_j C_{i,j} + M_i,
\end{equation}
where $\lambda_j>0$, the $C_{i,j}$ are irreducible curves generating a $K_{X_i}$-negative extremal in $\NA{X_i}$ and  where $M_i \in \NA{X_i}$ is
a class satisfying $K_{X_i} \cdot M_i \geq 0$. 

Notice now the following: since $K_{X_i}+B_i$ is $\Q$-linearly equivalent to an effective divisor with support $B_i$
and $(K_{X_i}+B_i)^3=0$ we have
$$
(K_{X_i}+B_i)^2 \cdot B_{i,l}=0
$$ 
for all $l$. In particular $(K_{X_i}+B_i) \cdot Z_{i,1}=0$. Since $K_{X_i}+B_i$ is nef we deduce from \eqref{decomposezi} that
\begin{equation} \label{trivialonray}
(K_{X_i}+B_i) \cdot C_{i,1}= 0.
\end{equation}
Let now $R_i$ be the extremal ray generated by $C_{i,1}$.
As in the proof of Lemma \ref{lemmaprepareMMP}, the locus of $R_i$ is contained in a surface with slc singularities (a component of $B_i$). 
Thus the contraction of $R_i$ exists by Theorem \ref{theoremcontraction} and, if the extremal ray is small, the
flip exists by Theorem \ref{theoremexistenceflip}. We denote by
$$
\varphi_i: (X_i, 0) \dashrightarrow (X_{i+1}, 0)
$$
the flip or divisorial contraction of $R_i$.  As in the proof of Lemma \ref{lemmaMMP1} we see that the properties $(1)$ and $(2)$ 
hold at level $i+1$.

By the induction hypothesis, the divisor $K_{X_i}+B_i$ is nef of numerical dimension two and by \eqref{trivialonray} numerically trivial on the extremal ray $R_i$.
Using Lemma \ref{lemmacrepantMMP} this implies that $K_{X_{i+1}}+B_{i+1}$ is nef
of numerical dimension two, moreover the pair $(X_{i+1}, B_{i+1})$ is lc.
Finally Lemma \ref{lemmaisomorphism} implies that 
$(X_{i+1} \setminus B_{i+1}) \simeq (X_i \setminus B_i)$
has terminal singularities. 

{\em Part III}
In total we have constructed a sequence of $K$-negative contractions satisfying the properties
(1)-(5). By Theorem \ref{theoremKtermination} any such sequence terminates after $n$ steps.
By construction this yields
$$
(K_{X_{n+1}}+B_{n+1}) \cdot K_{X_{n+1}}^2 \geq 0.
$$ 
Since all the contractions are $K+B$-trivial, the Kodaira dimension is invariant, i.e. we have
$\kappa(K_{\hat X}+ \hat B) = \kappa(K_{X_{n+1}}+B_{n+1})$.
\end{proof}


\section{Positivity of cotangent sheaves} \label{sectionpositivity}

In this section we briefly review stability and Chern classes on singular complex spaces,
then we prove the crucial Chern class inequality Theorem \ref{theoremnonnegative}.
For simplicity of notations we restrict ourselves to the case of threefolds with isolated singularities - which is all we need later - but all the statements
can easily be adapted for spaces of arbitrary dimension that are smooth in codimension two.

\begin{definition} \label{definitionnumbers}
Let $X$ be a normal compact K\"ahler threefold with isolated singularities.
Let $\sF_1$ and $\sF_2$ be coherent sheaves on $X$, and let $\holom{\pi}{\hat X}{X}$ be a log-resolution. 

Then the Chern classes $c_i(\pi^* \sF_1)$ and $c_i(\pi^* \sF_2)$ are well-defined elements of $H^{2i}(\hat X, \Z)$ \cite{ToTo86}, see also \cite{Gri10}. Thus for every $\alpha \in N^1(X)$ the intersection numbers
$$
\pi^* \alpha  \cdot c_1(\pi^* \sF_1) \cdot c_1(\pi^* \sF_2) \qquad \mbox{and} \qquad
\pi^* \alpha  \cdot c_2(\pi^* \sF_1) 
$$
are well-defined, so by the duality $N_1(X) = N^1(X)^*$ we define
$$
c_1(\sF_1) \cdot c_1(\sF_2) \in N_1(X): \alpha \mapsto \pi^* \alpha  \cdot c_1(\pi^* \sF_1) \cdot c_1(\pi^* \sF_2)
$$
and
$$
c_2(\sF_1) \in N_1(X): \alpha \mapsto \pi^* \alpha  \cdot c_2(\pi^* \sF_1).
$$
\end{definition}

\begin{lemma} \label{lemmaexactsequence}
In the situation of Definition \ref{definitionnumbers}, the classes $c_1(\sF_1) \cdot c_1(\sF_2)$ and $c_2(\sF_1)$ do not depend on the choice of the desingularisation.
Moreover if 
$$
0 \rightarrow \sF \rightarrow \sG \rightarrow \sQ \rightarrow 0
$$
is an exact sequence of coherent sheaves, then we have the usual formula
$$
c_2(\sG) = c_2(\sF) + c_2(\sQ) + c_1(\sF) \cdot c_1(\sQ).
$$
\end{lemma}

\begin{proof}

In order to see for instance that the definition $c_2(\sF)$ does not depend on the resolution it suffices to consider the case where
$\holom{\pi_1}{\hat X_1}{X}$ and $\holom{\pi_2}{\hat X_2}{X}$ are two log-resolutions with a factorisation
$\holom{q}{X_2}{X_1}$.  Then we want to show that
$$
(q^* \pi_1^* \alpha) \cdot c_2(q^* \pi_1^* \sF) = (\pi_1^* \alpha) \cdot c_2(\pi_1^* \sF).
$$ 
This follows from the general fact, that, given a holomorphic map $f: X\to Y$ between compact complex manifolds and 
$\sG$ a coherent sheaf on $Y,$ then $c_j(f^*(\sG)) = f^*(c_j(\sG)),$ cp. e.g. \cite{Gri10}.

Next consider a log-resolution $\holom{\pi}{\hat X}{X}$ with exceptional locus $D = \sum D_k$ and $\sF$ a coherent sheaf with support on $D$.
Let $i^k : D_k \to X$ be the inclusion. 
Then by the Grothendieck-Riemann-Roch formula, see e.g., \cite[Thm1.1]{Gri10}  (and \cite[Thm.15.2]{Ful98} in the algebraic case), 
$$
\mbox{ch}(i^k_* (\sF|_{D_k}) = i^k_* (\mbox{td}(N_{D_k/X})^{-1} \cdot \mbox{ch}(\sF|_{D_k})).
$$
Thus if $\alpha \in N^1(X)$, then the projection formula gives 
$$\pi^* \alpha \cdot c_2(i^k_*(\sF|_{D_k}) = 0.$$
 Now an easy induction on the number of components of $D$ shows that 
 \begin{equation} \label{chernfor}
 \pi^*(\alpha) \cdot c_2(\sF)  = 0
 \end{equation}

For the second statement we consider the pull-back 
$$
\pi^* \sF \stackrel{\alpha}{\rightarrow} \pi^* \sG \rightarrow \pi^* \sQ \rightarrow 0
$$
by a log-resolution $\holom{\pi}{\hat X}{X}$. In general the morphism $\alpha$ is not injective, but its kernel has support in the $\pi$-exceptional locus. 
Using the usual rules for computing Chern classes on compact manifolds and invoking (\ref{chernfor}), we have 
$$\pi^* \alpha \cdot c_2(\pi^* \sF) =  \pi^* \alpha \cdot c_2(\pi^* \sF/\ker \alpha)$$ 
and $$\pi^* \alpha \cdot c_1(\pi^* \sF) \cdot c_1(\pi^* \sQ) =  \pi^* \alpha \cdot c_1(\pi^* \sF/\ker \alpha) \cdot c_1(\pi^* \sQ).$$
Thus the statement follows from the standard formula in the smooth case. 
\end{proof}

\begin{remark} \label{torsion} In the proof of Lemma \ref{lemmaexactsequence} we have shown the following. 
Let $\pi: \hat X\to X$ be a log resolution of the K\"ahler threefold $X$ with only isolated singularities. Let $\sF$ be a coherent sheaf on $\hat X$,
supported on the exceptional divisor $D$ of $\pi$ and 
let $\alpha \in N^1(X).$ Then we have 
$$\pi^*(\alpha) \cdot c_2(\sF) = \pi^*(\alpha) \cdot c_1(\sF)^2 = 0.$$
\end{remark}

Given a normal compact K\"ahler threefold $X$ with isolated singularities and a torsion-free sheaf $\sF$ there is no obvious candidate for 
the first Chern class $c_1(\sF) \in H^2(X, \R)$. However we can define, as in the situation of Definition \ref{definitionnumbers}, for every $\alpha \in N^1(X)$ the intersection number
$\alpha^2 \cdot c_1(\sF)$ by pulling-back to some log-resolution $\pi: \hat X \to X$:
$$ 
\alpha^2 \cdot c_1(\sF)  := (\pi^*(\alpha))^2 \cdot c_1(\pi^*(\sF)).
$$ 

\begin{definition}
Let $X$ be a normal compact K\"ahler threefold with isolated singularities 
and let $\alpha$ be a nef class on $X$. We say that a non-zero torsion-free sheaf $\sF$ is $\alpha$-semistable (resp. $\alpha$-stable)
if for every non-zero saturated subsheaf $\sE \subset \sF$ we have
$$
\mu_{\alpha}(\sE) := \frac{\alpha^{2} \cdot c_1(\sE) }{\rk \sE} \leq \frac{\alpha^{2} \cdot c_1(\sF) }{\rk \sF}  =: \mu_{\alpha}(\sF) \qquad (\mbox{resp.} \ <). 
$$
\end{definition}

\begin{proposition} (Harder-Narasimhan filtration)
Let $X$ be a normal compact K\"ahler threefold with isolated singularities, 
and let $\alpha$ be a nef class on $X$. Let $\sF$ be a non-zero torsion-free coherent sheaf on $X$. Then there exists a filtration
$$
0 = \sF_0 \subset \sF_1 \subset \ldots \subset \sF_k = \sF
$$
such that for every $i \in \{1, \ldots, k\}$ the quotient $\sF_i/\sF_{i-1}$ is $\alpha$-semistable and we have a strictly decreasing sequence
of slopes
$$ \mu_{\alpha}(\sF_{i}/\sF_{i-1}) > \mu_{\alpha}(\sF_{i+1}/\sF_{i}) 
\qquad \forall i \in {1, \ldots, k-1}.
$$
\end{proposition} 

\begin{proof} We proceed as in \cite{Kob87}; the main point is to show that there is a constant $C$ such that 
$$ \mu_{\alpha}(\sF) \leq C $$
for all $\sF \subset \sE.$ 
First we reduce to the smooth case: take a log resolution $\pi: \hat X \to X$ such that $\pi^*(\sE)^{**} $ is locally free  and observe that 
$\mu_{\alpha}(\sF) = \mu_{\pi^*(\alpha)}(\pi^*(\sF) / {\rm tor})).$ 

So we may assume $X$ smooth and $\sE$ locally free. Then the proof of \cite[Lemma 7.16]{Kob87} works. 
Once the boundedness is settled, the existence of the filtration is shown as in the classical case. 
\end{proof}

We shall also use the following elementary result, cp. \cite[Thm.7.18]{Kob87}. 
\begin{proposition} (Jordan-H\"older filtration)
Let $X$ be a normal compact K\"ahler threefold with isolated singularities, 
and let $\alpha$ be a nef class on $X$. Let $\sF$ be a non-zero $\alpha-$semi-stable torsion-free coherent sheaf on $X$. Then there exists a filtration
$$
0 = \sF_0 \subset \sF_1 \subset \ldots \subset \sF_k = \sF
$$
such that for every $i \in \{1, \ldots, k\}$ the quotient $\sF_i/\sF_{i-1}$ is $\alpha$-stable.
\end{proposition} 

\begin{definition} \label{definitiongenericallynef}
Let $X$ be a normal compact K\"ahler threefold,
and let $\alpha$ be a nef class on $X$. A non-zero torsion-free coherent sheaf $\sF$ on $X$
is $\alpha$-generically nef if for every torsion-free quotient sheaf
$\sF \rightarrow Q \rightarrow 0$ we have
$$
\alpha^{2} \cdot c_1(Q) \geq 0.
$$
\end{definition}

\begin{remark} \label{remarkgenericallynef}
In the situation of Definition \ref{definitiongenericallynef}, let 
$$
0 = \sF_0 \subset \sF_1 \subset \ldots \subset \sF_k = \sF
$$
be the Harder-Narasimhan filtration with respect to $\alpha$. If $\sF$ is $\alpha$-generically nef, then by definition
$\alpha^{2} \cdot c_1(\sF/\sF_{k-1}) \geq 0$. Thus
$$
\alpha^{2} \cdot c_1(\sF_i/\sF_{i-1}) \geq 0
$$
for every $i \in \{1, \ldots, k\}$.
\end{remark}

We can now prove the Bogomolov inequality for {\em stable} sheaves on a singular space.

\begin{theorem} \label{theorembogomolov}
Let $X$ be a normal compact K\"ahler threefold with isolated singularities,
and let $\alpha$ be a K\"ahler class on $X$.  Let $\sF$ be an $\alpha-$stable non-zero torsion-free coherent sheaf on $X$.
Then we have
$$
\alpha \cdot c_2(\sF)  \geq \big(\frac{r-1}{2r}\big)  \alpha \cdot  c_1^2(\sF).
$$ 
\end{theorem}

\begin{proof}
We fix a log-resolution \holom{\pi}{\hat X}{X}.

{\em Step 1. Suppose that $\sF$ is reflexive.}
Set $\hat \sF := (\pi^* \sF)^{**}$. Since $\sF$ is reflexive, the sheaves $\hat \sF$ and $\pi^* \sF$ coincide in the complement of the $\pi$-exceptional locus.
In particular by Remark \ref{torsion} one has 
$$
\alpha \cdot c_2(\sF) = \pi^* \alpha \cdot c_2(\pi^* \sF) = \pi^* \alpha \cdot c_2(\hat \sF)
$$
and 
$$
\alpha \cdot  c_1^2(\sF) =  \pi^* \alpha \cdot  c_1^2(\pi^* \sF) =  \pi^* \alpha \cdot  c_1^2(\hat \sF).
$$
Thus it is sufficient to prove the inequality for $\hat \sF$. Arguing exactly as in the proof of \cite[Prop.6.9]{DP03} we see that $\hat \sF$ is $\pi^* \alpha$-stable.
Since stability is an open property \cite[Prop.2.1]{Cao13} we obtain that $\hat \sF$ is $(\pi^* \alpha + \varepsilon \omega)$-stable where $\omega$ is a K\"ahler form on $\hat X$
and $0 < \varepsilon \ll 1$. Thus \cite{BS94} yields
$$
(\pi^* \alpha + \varepsilon \omega) \cdot c_2(\hat \sF)  \geq \frac{r-1}{2r}  (\pi^* \alpha + \varepsilon \omega) \cdot  c_1^2(\hat \sF)
$$
for every $0 < \varepsilon \ll 1$. The claim follows by passing to the limit $\varepsilon \to 0$.

{\em Step 2. Reduction to the case where $\sF$ is reflexive.}
Since $\sF$ is torsion-free, the injection $i: \sF \hookrightarrow \sF^{**}$ is an isomorphism in codimension one. 
Thus the kernel and cokernel of 
$$\pi^*( \sF) \rightarrow \pi^* (\sF^{**})$$ have support in the union of the $\pi$-exceptional locus 
and a set of dimension at most one. In particular $c_1(\pi^* \sF) = c_1(\pi^* \sF^{**})+D$ with $D$ a $\pi$-exceptional divisor,
so the projection formula yields $$\alpha \cdot  c_1^2(\sF)= \alpha \cdot  c_1^2(\sF^{**}).$$
Thus we are done if we prove that $\alpha \cdot c_2(\sF^{**}) \geq \alpha \cdot c_2(\sF)$.
By the second part of Lemma \ref{lemmaexactsequence} it is sufficient to prove that
$$
\alpha \cdot c_2(\sF^{**}/\sF)  = \pi^*(\alpha) \cdot c_2(\pi^*(\sF^{**} / \sF))\geq 0.
$$
Let $S$ be the union of the $1-$dimensional irreducible components of the support of $\sF^{**} / \sF$ and $\hat S$ the strict transform in
$\hat X$; we may assume that the irreducible components $\hat S_i$ of $\hat S$ are smooth. 
Set
$$
Q := (i_{\hat S})_*(\pi^*(\sF^{**} /\sF)|_{\hat S}).
$$
Then by Remark \ref{torsion} and Lemma \ref{lemmaexactsequence}, it suffices to show 
$$ \pi^*(\alpha) \cdot c_2(Q) \geq  0.$$ 
Yet $Q$ has support on a set of dimension one, so 
the Grothendieck Riemann-Roch formula yields
$$
c_2(Q) = \sum a_i \hat S_i,
$$
where $a_i \in \N_0$. Since $\alpha$ is nef, the statement follows.
\end{proof}

\begin{lemma} \label{lemmahodgeindex}
Let $X$ be a compact K\"ahler threefold with isolated singularities. Let $\alpha$ be a nef class on $X$, and let $\sF$ be a
coherent sheaf. Then we have
$$
\left( \alpha^2 \cdot  c_1(\sF)  \right)^2 \geq \left(\alpha  \cdot  c_1^2(\sF) \right) \cdot \alpha^3.
$$
\end{lemma} 

\begin{proof}
By our definition of the intersection numbers \ref{definitionnumbers} we may suppose $X$ smooth. Since $\alpha$ is a limit of K\"ahler classes, 
it is sufficient to prove the statement under the stronger hypothesis that $\alpha$ is a K\"ahler class. Now the inequality
follows from the usual Hodge index theorem \cite[Thm.6.33]{Voi02} by an elementary computation.
\end{proof}

\begin{theorem} \label{theoremnonnegative}
Let $(X, \omega)$ be a compact K\"ahler threefold with isolated singularities, and let $\sF$ be a non-zero reflexive coherent sheaf
on $X$ such that $\det \sF$ is $\Q$-Cartier. Suppose that there exists a pseudoeffective class $P \in N^1(X)$ such that
$$
L := c_1(\sF) + P
$$
is a nef class. Suppose furthermore that for all $0 < \varepsilon \ll 1$ the sheaf $\sF$ is $(L+\varepsilon \omega)$-generically nef. Then we have
$$
L \cdot c_2(\sF) \geq \frac{1}{2} (L \cdot c_1^2(\sF) - L^3).
$$
In particular, if $L \cdot c_1^2(\sF) \geq 0$ and $L^3=0$, then 
\begin{equation} \label{nonnegative}
L \cdot c_2(\sF) \geq 0.
\end{equation}
\end{theorem}

\begin{remark*} This statement (and its proof) is a variation of \cite[Thm.6.1]{Miy87}, \cite[Prop.10.12]{Uta92}. However the weaker assumptions will
be crucial for the application in the proof of Theorem \ref{theoremnutwo}.
\end{remark*}

\begin{proof}
 Fix  $0 < \varepsilon \ll 1$, and consider the Harder-Narasimhan filtration 
$$ 
0 = \sF_0 \subset \sF_1 \subset \ldots \subset \sF_k = \sF 
$$
for $\sF$ with respect to $\sLe:=L+\varepsilon \omega$. Then by Lemma \ref{lemmaexactsequence} 
\begin{equation} \label{c2one}
\sLe \cdot c_2(\sF) = \sLe \cdot 
\left(
\sum_{1 \leq i < j \leq k} c_1(\sF_i/\sF_{i-1}) c_1(\sF_j/\sF_{j-1}) + 
\sum_{i=1}^k c_2(\sF_i/\sF_{i-1})
\right).
\end{equation}
Since $\sF$ is $\sLe$-generically nef, by Remark \ref{remarkgenericallynef}
\begin{equation} \label{nonnegative1}
\sLef \cdot c_1(\sF_i/\sF_{i-1}) \geq 0 \qquad \forall \ i \in \{1, \ldots, k \}.
\end{equation} 
For every $i \in \{ 1, \ldots, k \}$, let
$$
0 = \sF_{i,0} \subset \sF_{i, 1} \subset \ldots \subset \sF_{i, k_i} = \sF_i/\sF_{i-1} 
$$
be the Jordan-H\"older filtration of $\sF_i/\sF_{i-1}$  with respect to $\sLe$. Then we have
\begin{equation} \label{c2two}
\sLe \cdot c_2(\sF_i/\sF_{i-1} ) = \end{equation} 
$$ = \sLe \cdot 
\left(
\sum_{1 \leq p < q \leq k_i} c_1(\sF_{i,p}/\sF_{i,p-1}) c_1(\sF_{i, q}/\sF_{i, q-1}) + 
\sum_{p=1}^{k_i} c_2(\sF_{i,p}/\sF_{i,p-1})
\right). $$
Since $\sF_i/\sF_{i-1}$ is $\sLe$-semistable with non-negative slope by \eqref{nonnegative1}, we obtain 
\begin{equation} \label{nonnegative2}
\sLef \cdot c_1(\sF_{i,p}/\sF_{i,p-1}) \geq 0 \qquad \forall \ i \in \{1, \ldots, k \}, p \in \{ 1, \ldots, k_i\}.
\end{equation} 
Plugging the equations \eqref{c2two} into the equation \eqref{c2one} and following the lexicographic order,
we rename the graded pieces $\sF_{i,p}/\sF_{i,p-1}$ into $\sG_l$, where $l \in \{ 1, \ldots, n \}$. Thus
\begin{equation} \label{decomp}
\sLe \cdot c_2(\sF) = \sLe \cdot 
\left(
\sum_{1 \leq l < m \leq n} c_1(\sG_l) c_1(\sG_m) + 
\sum_{l=1}^n c_2(\sG_l)
\right),
\end{equation}
where the $\sG_l$ are non-zero torsion-free $\sLe-$stable sheaves. Moreover, by \eqref{nonnegative2} we have
\begin{equation} \label{nonnegative3}
\sLef \cdot c_1(\sG_l) \geq 0 \qquad \forall \ l \in \{1, \ldots, n \}. 
\end{equation} 
Since we renamed
according to the lexicographic order, the sequence of slopes with respect to $\sLe$ is (not necessarily strictly) decreasing.
Thus, setting $r_l := \rk \sG_l$ and  
$$
\alpha_l := 
{{\sLef \cdot c_1(\sG_l)} \over {r_l \sLee}},
$$
we obtain
\begin{equation} \label{decreasing}
\alpha_1 \geq \alpha_2 \geq \ldots \geq \alpha_n \geq 0.
\end{equation}
By Theorem \ref{theorembogomolov} we have
$\sLe \cdot c_2(\sG_l) \geq  \big( {{r_l-1} \over {2r_l}} \big) \ \sLe \cdot c_1^2(\sG_l)$ for all $l \in \{1, \ldots, n \}$. 
Therefore using \eqref{decomp} we obtain
\begin{eqnarray*}
\sLe \cdot c_2(\sF)
&=&
\sLe \cdot  
\left(
\frac{1}{2} c_1(\sF)^2
+
\sum_{l=1}^n c_2(\sG_l)
-
\frac{1}{2}
\sum_{l=1}^n c_1(\sG_l)^2 
\right)
\\
& \geq &
\sLe \cdot  
\left(
\frac{1}{2} c_1(\sF)^2
-
\frac{1}{2 r_l}
\sum_{l=1}^n c_1(\sG_l)^2 
\right).
\end{eqnarray*}
By Lemma \ref{lemmahodgeindex} we have 
$$
(\sLef \cdot c_1(\sG_l))^2 \geq  (\sLe \cdot c_1^2(\sG_l)) \cdot \sLee \qquad \forall \ l \in \{ 1, \ldots, n \},
$$
so, using the coefficient $\alpha_l$ defined above, we get 
$
 \sLe \cdot c_1^2(\sG_i) \leq \alpha_i^2 r_i^2 \sLee.
$
Putting this into the last inequality yields
\begin{eqnarray}
\sLe \cdot c_2(\sF) &\geq&
\sLe \cdot  
\left(
\frac{1}{2} c_1(\sF)^2
-
\frac{1}{2}
\sum_{l=1}^n \alpha_l^2 r_l \sLef. 
\right)
\\
& = &
\label{helpme}
\frac{1}{2}  \sLe \cdot  
\left(
(c_1(\sF)^2
-
\sLef)
+
(1-
\sum_{l=1}^n \alpha_l^2 r_l) \sLef. 
\right)
\end{eqnarray}
{\it We claim that} $$(1- \sum_{l=1}^n \alpha_l^2 r_l) \geq 0.$$ 
Assuming this for the time being, let us see how to conclude :
since $\sLee>0$, the {\it claim} together with \eqref{helpme} yields
$$
\sLe \cdot c_2(\sF) \geq \frac{1}{2}  \sLe \cdot  
\left(
c_1(\sF)^2
-
\sLef
\right).
$$
Now we take the limit $\varepsilon \rightarrow 0$.

{\em Proof of the claim.}
First of all, \eqref{decreasing} implies that 
$$
1 - \sum_{l=1}^n \alpha_l^2 r_l   \geq 1 - \alpha_1 \sum_{l=1}^n r_l \alpha_l. 
$$
Now by definition of the  numbers $\alpha_l$ we have
$$
 \sum_{l=1}^n r_l \alpha_l =  \sum_{l=1}^n {{\sLef \cdot c_1(\sG_l)} \over { \sLee}}
 = \frac{\sLef \cdot c_1(\sF)}{\sLee}.
$$
Yet by definition of $L$ and $\sLe$, 
$$
c_1(\sF)=L-P=\sLe-P-\varepsilon \omega.
$$
Since $P$ is pseudoeffective and since $\omega$ is K\"ahler it follows that $\sLef \cdot c_1(\sF) \leq \sLee$, hence
$\sum_{l=1}^n r_l \alpha_l \leq 1.$
Since $r_l \geq 1$ and $\alpha_l \geq 0$ for all
$l \in \{ 1, \ldots, n\}$ this implies $\alpha_1 \leq 1$, proving the {\it claim}. 
\end{proof}



\section{Abundance} \label{sectionabundance}

In this section, we establish the abundance theorem for non-algebraic K\"ahler threefolds. The main difficulty is to show that if 
the numerical Kodaira dimension $\nu(X) = 1$ or $\nu (X) = 2,$ then $\kappa (X) \geq 1. $ This will be done in Theorems \ref{theoremnuone} and \ref{theoremnutwo}.

\subsection{The case $\nu=1$}

\begin{theorem} \label{theoremnuone}
Let $X$ be a normal $\Q$-factorial compact  K\"ahler threefold with at most terminal singularities
such that $K_X$ is nef. If $\nu(X)=1$, then $\kappa(X) \geq 1$.
\end{theorem}

\begin{proof}
By Lemma \ref{lemmaMMP1} and Lemma \ref{lemmanuoneMMP} we are reduced to proving the following statement:

{\em
Let $X$ be a normal $\Q$-factorial compact K\"ahler threefold such that $(X,0)$ is klt,  carrying 
a divisor $D \in |m K_{X}|$ with the following properties:
\begin{enumerate}
\item Set $B:=\supp D$. The pair $(X, B)$ is lc and $X \setminus B$ has terminal singularities.
\item The divisor $K_{X}+B$ is nef and we have $\nu(K_{X}+B)=1$. 
\item There exists an irreducible component $S \subset B$ that is a connected component. 
\end{enumerate}
Then $\kappa(X) \geq 1$.
}

We follow the argument in \cite[Ch.13]{Uta92}: by adjunction
\cite[16.9.1]{Uta92} there exists a boundary divisor $\Delta$
on $S$ such that $(S, \Delta)$ is slc and
$
K_{S}+\Delta
$
is numerically trivial. By Proposition \ref{propositionslczero} the divisor  $K_{S}+\Delta$ is torsion.
Due to a covering trick of Miyaoka \cite[11.3.6]{Uta92} we may suppose
after a finite cover, \'etale in codimension one, that
$$
\omega_{S} = \sO_S(K_S)  \simeq \sO_{S} \simeq \sO_{S}(S).
$$
Note that by \cite[Prop.5.20(4)]{KM98} the klt property is preserved under a finite morphism which is \'etale in codimension one.
In particular, $X$ is Cohen-Macaulay. By \cite[11.3.7]{Uta92} we are finished if we prove that
for every infinitesimal neighbourhood $S_n$, the restriction
$$
H^p(S_n, \sO_{S_n}) \rightarrow H^p(S, \sO_S)
$$
is surjective for every $p \in \N$.
Observe here that condition (3) in \cite[11.3.7]{Uta92}
is satisfied by \cite[12.1.2]{Uta92}, as explained in \cite[p.158]{Uta92}.
In fact, using the commutative diagram 
 $$
\xymatrix{
H^p(S_n, \C) \ar[d] \ar[r]  & H^p(S_n, \sO_{S_n}) \ar[d]
\\
H^p(S, \C) \ar[r] & H^p(S, \sO_S)
}
$$
and the isomorphism $H^p(S_n, \C) \simeq H^p(S, \C)$ we see that it is sufficient to prove that
$$
H^p(S, \C) \rightarrow H^p(S, \sO_S)
$$ 
is surjective. This is done in Theorem \ref{theoremslcsurjective}.
\end{proof}

\subsection{The case $\nu = 2$}

\begin{theorem} \label{theoremnutwo}
Let $X$ be a normal $\Q$-factorial compact  K\"ahler threefold with at most terminal singularities
such that $K_X$ is nef. If $\nu(X)=2$, then $\kappa(X) \geq 1$.
\end{theorem}

The basic idea of the proof is the same as in the projective case \cite[Sect.14]{Uta92},
however the computations get considerably simplified by our generalisation of Miyaoka's Chern class inequality \cite[Thm.6.1]{Miy87}.
Let us start by recalling the Riemann-Roch formula for terminal threefolds:

\begin{proposition} \label{propositionRR}
Let $X$ be a normal compact  K\"ahler threefold with at most terminal singularities, and let $L$ be a line bundle on $X$.
Then we have
$$
\chi(X, L) = \frac{L^3}{6} - \frac{1}{4} K_X \cdot L^2 + L \cdot \frac{K_X^2+c_2(X)}{12} + \chi(X, \sO_X),
$$
where $L \cdot c_2(X) := \pi^* L \cdot c_2(\hat X)$ with $\holom{\pi}{\hat X}{X}$ 
any resolution of singularities (cf. Definition \ref{definitionnumbers}).
\end{proposition} 

\begin{proof}
Let $\holom{\pi}{\hat X}{X}$ be a resolution of singularities, which is an isomorphism over the smooth locus of $X.$ Since $X$ has rational singularities, we have
$\chi(X, L) = \chi(\hat X, \pi^* L).$
Riemann-Roch on the smooth K\"ahler threefold $\hat X$ yields
\begin{equation} \label{RR}
\chi(\hat X, \pi^* L) = \frac{\pi^* L^3}{6} - \frac{1}{4} K_{\hat X} \cdot (\pi^*L)^2 + \pi^*L \cdot \frac{K_{\hat X}^2+c_2(\hat X)}{12} + \chi(\hat X, \sO_{\hat X}).
\end{equation} 
Now $\pi^* L^3=L^3$, and, using again  the rationality of the singularities of $X$, $\chi(\hat X, \sO_{\hat X})=\chi(X, \sO_X)$. Since $X$ is smooth in codimension
two we may write $$K_{\hat X} \sim_\Q \pi^* K_X + E$$ with $E$ a divisor such that $\pi(E)$ is finite. In particular the projection formula gives 
$K_{\hat X} \cdot (\pi^*L)^2=K_X \cdot L^2$ and $\pi^*L \cdot K_{\hat X}^2 = L \cdot K_X^2$. Thus \eqref{RR} gives our claim. 
\end{proof}

We will also need the following K\"ahler version of Miyaoka's generic nefness theorem, due to Enoki in the smooth case:

\begin{proposition} \label{propositiongenericallynef}
Let $X$ be a normal compact K\"ahler space of dimension $n$ with canonical singularities. 
Suppose that $K_X$ is nef or $\kappa(X) \geq 0$.
Then $\Omega_X$ is generically nef with respect to any nef class $\alpha$, i.e. for every torsion-free quotient sheaf
$$
\Omega_X \rightarrow \sQ \rightarrow 0,
$$
we have $\alpha^{n-1} \cdot c_1(\sQ) \geq 0$.
\end{proposition} 

\begin{proof} 
Fix a nef class $\alpha \in N^1(X)$, and let $\Omega_X \rightarrow \sQ \rightarrow 0$ be a torsion-free quotient sheaf.
Let $\pi: \hat X \to X$ be a desingularisation by a compact K\"ahler manifold, and let
$\sK $ be the kernel of the induced epimorphism 
$$
(\pi^* \Omega_X) / {\rm torsion} \to (\pi^* \sQ) / {\rm torsion} \to 0. 
$$
Using the injective map $$(\pi^* \Omega_X) / {\rm torsion} \hookrightarrow \Omega_{\hat X}$$ we may view $\sK$ as a subsheaf
of $\Omega_{\hat X}$, and we denote by $\hat \sK$ its saturation in $\Omega_{\hat X}$. Set 
$$
\hat \sQ := \Omega_{\hat X}/\hat \sK,
$$
then $\hat \sQ$ is a torsion-free quotient of $\Omega_{\hat X}$ coinciding with $(\pi^* \sQ) / {\rm torsion}$
in the complement of the exceptional locus. In particular we have
$$
\alpha^{n-1} \cdot c_1(\sQ) = (\pi^* \alpha)^{n-1} \cdot c_1((\pi^* \sQ) / {\rm torsion}) = (\pi^* \alpha)^{n-1} \cdot c_1(\hat \sQ).
$$
We will now verify the conditions of \cite[Thm.1.4]{Eno88} in order to conclude that $(\pi^* \alpha)^{n-1} \cdot c_1(\hat \sQ) \geq 0$:
since $X$ has canonical singularities we have
$$
K_{\hat X} = \pi^* K_X + E,
$$
with $E$ an effective, $\pi$-exceptional $\Q$-divisor. \\
Thus if $K_X$ is nef the conditions of \cite[Thm.1.4]{Eno88} hold
by setting $L= \pi^* K_X$ and $D=E$. If $\kappa(X) \geq 0$ we have $\pi^* K_X = F$ with $F$ an effective $\Q$-divisor, so 
the conditions  are satisfied by setting $L=0$ and $D=F+E$. In both cases Enoki's theorem tells us that
$$
\omega^{n-1} \cdot c_1(\hat \sQ) \geq 0
$$
for every K\"ahler form $\omega$ on $\hat X$. Since $\pi^* \alpha$ is nef, the statement follows by passing to the limit.
\end{proof} 

We will use the following Serre vanishing property:

\begin{lemma} \label{lemmaserre}
Let $X$ be a normal $\Q$-factorial compact K\"ahler threefold, and let $B_1, \ldots, B_k$ be prime Weil divisor on $X$
such that $B_i$ is Cohen-Macaulay for every $i \in \{1, \ldots, k \}$. Let $L$ be a nef Cartier divisor on $X$ such that
$L|_{B_i} \not \equiv 0$ for every $i \in \{1, \ldots, k \}$. Let $Y \subset X$ be a subscheme such that $Y_{\red} \subset \sum_{i=1}^k B_i$
and let $\sF$ be a coherent sheaf on $Y$. 
Then there exists a number $n_0 \in \N$ such that
$$
H^2(Y, \sF \otimes \sO_Y(L^{\otimes n})) = 0
$$
for every $n \geq n_0$.
\end{lemma}

\begin{proof}
Note first that we may suppose that $Y$ is defined by an ideal sheaf $\sO_X(- \sum_{i=1}^k a_i B_i)$ with $a_i \in \N$. 
Indeed, at the general point of every surface $B_i$, the scheme $Y$ is isomorphic to a scheme $\sum_{i=1}^k a_i B_i$ 
defined by the ideal sheaf  $\sO_X(- \sum_{i=1}^k a_i B_i)$. Thus if we consider the restriction map
$$
\sF \twoheadrightarrow \sF \otimes \sO_{\sum_{i=1}^k a_i B_i},
$$
its kernel has support on a scheme of dimension at most one. Hence
$$
H^2(Y, \sF \otimes \sO_Y(L^{\otimes n}))  \simeq H^2(\sum_{i=1}^k a_i B_i, \sF \otimes \sO_{\sum_{i=1}^k a_i B_i}(L^{\otimes n})) 
$$
for every $n \in \N$. We will now argue by induction on $\sum_{i=1}^k a_i$. The start of the induction is the case where $\sF$ is a coherent sheaf
on one of the surfaces $B_i$. Since $B_i$ is Cohen-Macaulay, Serre duality gives 
$$
H^2(B_i, \sF \otimes \sO_{B_i}(L^{\otimes n})) \simeq 
\mbox{Hom}(\sF \otimes \sO_{B_i}(L^{\otimes n}), \omega_{B_i}).
$$
Let $\holom{\nu}{\tilde B_i}{B_i}$ denote the normalisation. Since $\Homsheaf(\sF, \omega_{B_i}) $ is torsion free, 
the natural map
$$
\Homsheaf(\sF, \omega_{B_i})  \rightarrow \nu_*(\nu^*(\Homsheaf(\sF, \omega_{B_i})) 
$$
is injective. Since $\nu_*(\nu^*(\Homsheaf(\sF, \omega_{B_i}))$ and
$\nu_*(\nu^*(\Homsheaf(\sF, \omega_{B_i})/Tor)$ coincide at the generic point of $B_i$ and $\Homsheaf(\sF, \omega_{B_i})$ 
is torsion free, the map
$$
\Homsheaf(\sF, \omega_{B_i})  \rightarrow \nu_*(\nu^*(\Homsheaf(\sF, \omega_{B_i})/Tor)
$$
is injective, too. Thus it is suffices to show that for any torsion-free sheaf $\sS$ on $\tilde B_i$ one has
$$ 
H^0(\tilde B_i, \sS \otimes \sO_{\tilde B_i}(\tilde L^{- \otimes n}) ) = 0
$$
for $n \gg 0$ and $\tilde L:=\nu^*(L|_{B_i})$. 
Passing to a desingularisation, we may assume $\tilde B_i$ smooth and $\sS$ locally free. 
Fix now a K\"ahler form $\omega$ on $\tilde B_i$, and let $\sG_1 \subset \sS$ be the first sheaf of the Harder-Narasimhan filtration
with respect to $\omega$.
Since $L|_{B_i}$ is a non-zero nef divisor we have $\tilde L \cdot \omega>0$. Thus there exists a number $n_0 \in \N$ such that
$\sG_1 \otimes  \sO_{\tilde B_i}(\tilde L^{- \otimes n})$ has negative slope for all $n \geq n_0$. 
In particular $\sS \otimes \sO_{B_i}(L^{- \otimes n})$
has no global sections.
 
For the induction step we simply choose a surface $B_j$ such that $a_j>0$. Then the kernel of the restriction map
$$
\sF \twoheadrightarrow \sF \otimes \sO_{(\sum_{i=1}^k a_i B_i)-B_j}
$$ 
is a sheaf $\sG$ with support on $B_j$. Thus by the induction hypothesis we know that for $n \gg 0$ the second cohomology vanishes
for both  $\sG \otimes \sO_{B_j}(L^{\otimes n})$ and $\sF \otimes \sO_{(\sum_{i=1}^k a_i B_i)-B_j}(L^{\otimes n})$.  
\end{proof} 

\begin{corollary} \label{corollaryserre}
Let $X$ be a normal $\Q$-factorial compact K\"ahler threefold, and let $L$ be a nef Cartier divisor on $X$.
Let $D \in |L|$ be effective and set $B:= \supp D$ and $B_1, \ldots, B_k$ for the irreducible components of $B$. Suppose 
that $B_i$ is Cohen-Macaulay for every $i \in \{1, \ldots, k \}$. 

Suppose also that $L|_{B_i} \not \equiv 0$ for every $i \in \{1, \ldots, k \}$.
Then there exists a number $n_0 \in \N$ and constants $c_1, c_2 \in \N$ such that
for all $n \geq n_0$:
$$
\dim H^2(X, L^{\otimes n}) = c_1
$$
and
$$
\dim H^2(X, L^{\otimes n} \otimes \sO_X(-B)) = c_2.
$$
\end{corollary}

\begin{proof}
For all $n \in \N$ we have an exact sequence 
$$
0 \rightarrow L^{\otimes n-1} \rightarrow L^{\otimes n} \rightarrow \sO_D(L^{\otimes n}) \rightarrow 0.
$$
By Lemma \ref{lemmaserre}, $$H^2(D, \sO_D(L^{\otimes n}))=0$$ for all $n\gg 0$. Thus the map
$H^2(X, L^{\otimes n-1}) \rightarrow H^2(X, L^{\otimes n})$ is surjective for all $n \gg 0$. This shows the first statement.

For the second statement assume without loss of generality that $D-B \neq 0$.
Note that $D-B$ is an effective Weil divisor whose support is contained in $B$. For all $n \in \N$ we have an exact sequence 
$$
0 \rightarrow L^{\otimes n-1} \rightarrow L^{\otimes n-1} \otimes \sO_X(D-B) \rightarrow \sO_{D-B}(L^{\otimes n-1}) \rightarrow 0.
$$
Again by Lemma \ref{lemmaserre} 
$$H^2(D-B, \sO_{D-B}(L^{\otimes n-1}))=0$$ for all $n\gg 0$. Thus the
map  $$H^2(X, L^{\otimes n-1}) \rightarrow H^2(X, L^{\otimes n-1} \otimes \sO_X(D-B))$$ 
is surjective for $n \gg 0$ and
by the first statement $H^2(X, L^{\otimes n-1})$ is constant for $n \gg 0$. We conclude by noting that
$$L^{\otimes n-1} \otimes \sO_X(D-B) \simeq L^{\otimes n} \otimes \sO_X(-B)$$ 
since both sheaves are reflexive and coincide on the smooth locus of $X$.
\end{proof}

\begin{proof}[Proof of Theorem \ref{theoremnutwo}]  
By Lemma \ref{lemmaMMP1} and Lemma \ref{lemmanutwoMMP} we are reduced to proving the following statement:

{\em
Let $X$ be a normal $\Q$-factorial compact K\"ahler threefold such that $(X,0)$ is klt, carrying
a divisor $D \in |m K_{X}|$ with the following properties:
\begin{enumerate}
\item Set $B:=\supp D$. The pair $(X, B)$ is lc and $X \setminus B$ has terminal singularities.
\item The divisor $K_{X}+B$ is nef and we have $\nu(K_{X}+B)=2$. Moreover we have
$
\kappa(X) = \kappa(K_X+B).
$
\item For every irreducible component $T \subset B$ we have $(K_{X}+B)|_{T} \neq 0$.
\item We have $(K_{X}+B) \cdot K_{X}^2 \geq 0$.
\end{enumerate}
Then $\kappa(X) \geq 1$.
}

{\em Step 1: Singularities of $X$.} 
 
We claim that there is a finite set $S \subset X$ such that $X \setminus S$ has canonical singularities. By hypothesis
$X \setminus B$ has only terminal singularities, which are isolated. Thus it remains to consider the singular points of $X$ which are contained 
in $B$. Taking a finite covering of $X$ by analytic neighbourhoods we see that it is sufficient to prove the claim for $X$ a Stein variety.
Thus we can take a hyperplane section $H$ of $X$. Now for general $H$, the pair $(H,B_H) $ is lc by \cite[Prop.7.7]{Kol95}, so by \cite[Thm.9.6]{Kaw88}
every point $p \in B_H \subset H$ is a rational double point in $H$.  Hence by \cite[Thm.5.34]{KM98}, the point $p$ is a canonical 
singularity of $X.$ 
 
{\em Step 2: A Chern class inequality.}

Let $\holom{\mu}{X'}{X}$ be a terminal modification of $X$ (cf. Theorem \ref{theoremterminalmodel}). 
Thus $X'$ has only terminal singularities, and 
there exists an effective $\Q$-divisor $\Delta$ such that
$$
K_{X'} + \Delta \sim_\Q \mu^* K_X.
$$ 

Let $m$ be the Cartier index of $K_{X}+B$, then 
$$ L := m(K_X + B)$$
is Cartier, and we set $L':= \mu^* (L)$.
We prove the basic Chern class inequality 
\begin{equation} \label{todd}
 L' \cdot (K_{X'}^2+c_2(X'))\geq 0.
\end{equation}
In fact, since $X$ has only finitely many
non-canonical points, $\mu(\Delta)$ is finite.
Therefore the projection formula and our assumption d) yield
\begin{equation} \label{liftnonnegative}
L' \cdot K_{X'}^2 = \mu^* L \cdot K_{X'}^2 = m (K_{X}+B) \cdot K_{X}^2  \geq 0.
\end{equation}
By Proposition \ref{propositiongenericallynef} the sheaf $\Omega_{X'}$ is generically nef.
Since
$$
K_{X'} + \Delta + \mu^* B \sim_\Q \mu^* (K_X+B)
$$ 
is nef, the conditions of Theorem \ref{theoremnonnegative} are satisfied for $\sF:=(\Omega_X)^{**}$ and $P:= \Delta + \mu^* B$.
Having in mind  that $L^3= m^3 (K_X+B)^3=0$ and using \eqref{liftnonnegative},
Theorem \ref{theoremnonnegative} yields 
$$ L'\cdot c_2(X') \geq 0,$$
hence the Chern class inequality (\ref{todd}) is established.

{\em Step 3: A Riemann-Roch computation.} 
Since $K_X+B$ is nef and $\Q$-linearly equivalent to an effective divisor with support $B$, the equality $(K_X+B)^3=0$
implies that 
\begin{equation} \label{zerosquare}
(K_X+B)^2 \cdot T=0
\end{equation}
for every  irreducible component $T \subset B$. Since $K_X$ is $\Q$-linearly equivalent to an effective divisor with support $B$,
we conclude 
$$(K_X+B)^2 \cdot K_X=0.$$
Since $\mu(\Delta)$ is finite the projection formula yields
$$
K_{X'} \cdot (L')^2 =  K_X \cdot L^2 - \Delta \cdot (\mu^*L)^2 =  m^2 K_X \cdot (K_X+B)^2 = 0.
$$
Thus Proposition \ref{propositionRR} gives 
$$
\chi(X', \sO_{X'}(nL')) = n L'\cdot \frac{K_{X'}^2+c_2(X')}{12} + \chi(X', \sO_{X'})
$$
for all $n \in \N$.
Thus \eqref{todd} yields a constant $k$ such that 
$$
\chi(X', \sO_{X'}(nL')) \geq  k 
$$
for all $n \in \N$. Since $X$ has rational singularities, we conclude that 
$$
\chi(X, \sO_X(nL)) \geq  k 
$$
for all $n \in \N$. Since $H^3(X, nL)=0$ for $n \gg 0$ and $\dim H^2(X, nL)$ is constant for $n \gg 0$ by Corollary \ref{corollaryserre} 
(note that the components of $B$ are Cohen-Macaulay by \cite[5.25]{KM98}),
we arrive at 
\begin{equation} \label{hzero}  
h^0(X, nL) \geq h^1(X, nL) + c
\end{equation}
with some constant $c \in \Z$.

{\em Step 4. A simple case.} 
Although not really necessary, it is instructive to give the simple concluding argument in the case of strict inequality in 
(\ref{todd}). Then the preceding computation yields that
$$
h^0(X, nL) \geq h^1(X, nL) + D n
$$
with some positive constant $D$. Thus 
$$
\kappa(K_X+B) = \kappa(X, L)>0,
$$ 
so $\kappa (X) \geq 1$ by assumption b).

{\em Step 5: Conclusion.}
By \eqref{hzero} it suffices to show that $h^1(X, nL)$ grows at least linearly. 
We consider the exact sequence
\begin{equation} \label{exseq}  
0 \to \sO_{X}(nL-B) \to \sO_{X}(nL) \to \sO_{B}(nL) \to 0.
\end{equation} 
By Corollary \ref{corollaryserre} we know that $h^2(X, \sO_{X}(nL-B))$ is constant for $n \gg 0$.
Taking cohomology  of the exact sequence \eqref{exseq}, it remains to show that $h^1(B, \sO_B(nL))$ grows at least linearly. 
To this extent, we will prove that $\chi(B, \sO_B(nL))$ is constant. 
Assuming this for the time being, let us see how to conclude: by Lemma \ref{lemmaserre} we have
$H^2(B, \sO_B(nL))=0$ for $n \gg 0$. Moreover by adjunction \cite[16.9.1]{Uta92}
$\sO_B(K_B) \simeq \sO_B(K_X+B)$, so by Proposition \ref{propositionslcone}
$$ 
h^0(B, \sO_B(nL)) = h^0(B, \sO_{B}(nm(K_{X} + B))) = h^0(B, \sO_B(n m K_B))
$$
grows linearly. Thus $h^1(B, \sO_B(nL))$ grows linearly.

{\em Proof of the claim.} 
By \cite[Thm.3.1]{LR13} the Euler characteristic $\chi(B, \sO_B(nL))$ on the slc surface $B$ is computed by the usual Riemann-Roch 
formula\footnote{We refer to \cite[Ch.2.3]{LR13} for the definition of the intersection product on the non-normal surface $B$.}
$$ 
\chi(B, \sO_B(nL)) = \chi(B,\sO_B) + \frac{1}{2} (nL|_B) \cdot (nL|_B - K_B).
$$ 
Yet by \eqref{zerosquare} we have $0 = L^2 \cdot B = (L|_B)^2$. Since $\sO_B(K_B) \simeq \sO_B(K_X+B)$ is a multiple of $L|_B$,
this also implies that $L|_B \cdot K_B=0$.
\end{proof} 

\subsection{Proof of Theorem \ref{theoremmain}}

\begin{proof} 
By \cite[Thm.0.3]{DP03} we have $\kappa(X) \geq 0$.
If $\nu(X) = 3,$ then $K_X$ is big, hence $X$ is Moishezon and therefore projective \cite{Nam02}. Thus the result follows from the base point free theorem. 
If $\kappa(X)=\nu(X) \leq 2$ the statement follows from Kawamata's theorem \cite[Thm.1.1]{Kaw85b}, \cite[Thm.5.5]{Nak85}, \cite[Sect.4]{Fuj11}.

By Theorem \ref{theoremnuone} and Theorem \ref{theoremnutwo} we are thus left to exclude the possibility that
$\kappa(X)=1$ and $\nu(X)=2$.  This is done exactly as in \cite[Thm.7.3]{Kaw85b}.
\end{proof}

\section{Applications} 

In this concluding section we apply the MMP to explore the structure of non-algebraic compact K\"ahler threefolds $X$ with Kodaira dimension $\kappa (X) \leq 0.$ 

\subsection{Uniruled threefolds} 
\begin{theorem} \label{appli}  Let $X$ be a smooth non-algebraic compact K\"ahler threefold with $\kappa (X) = - \infty.$ Then $X$ is bimeromorphic to a normal compact K\"ahler threefold
$X'$ with at most terminal singularities with the following properties. There exists a contraction
$$ \varphi: X' \to Y$$
of an extremal ray in $\overline{NA}(X')$ such that 
\begin{enumerate}
\item[(1)] $Y$ is a normal  non-algebraic K\"ahler surface with only rational singularities. 
\item[(2)] There is a finite set $A \subset Y$ such
that $\varphi \vert X' \setminus \varphi^{-1}(A) \to Y \setminus A$ is a conic bundle. 
\item[(3)] $\varphi$ realises the MRC fibration of $X'$, and $\kappa (\hat Y) \in \{0,1\} $ for any desingularisation $\hat Y \rightarrow Y$. 
\end{enumerate} 
\end{theorem} 

\begin{proof} 
By  \cite[Cor.1.4]{HP13a} the manifold $X$ is uniruled. We first claim that ``the'' base $B$, chosen smooth,  of the MRC-fibration $X \dashrightarrow B$ has dimension two.
Indeed, if $\dim B=1$, then the MRC fibration is realised by a morphism $X \rightarrow B$ and the general fibre $F$ is rationally connected,
hence $H^2(F, \sO_F)=0$. This immediately implies $H^2(X, \sO_X)=0$, so $X$ is projective by Kodaira's criterion. The same of course applies when $\dim B = 0$
which is to say that $X$ is rationally connected. Notice also that $B$ is non-algebraic, otherwise $X$ were algebraic. 

By \cite[Thm.1.1]{HP15} we can run a MMP which terminates with a Mori fibre space $\varphi: X' \to Y$. Moreover the properties $(1)$ and $(2)$
are shown in \cite[Thm.1.1]{HP15} and \cite[Rem.4.2]{HP15}. We have seen above that the base of the MRC fibration has dimension two,
so $\varphi$ realises this fibration for $X'$. In particular $Y$ is not algebraic, so any desingularisation $\hat Y$ is not uniruled.
Thus we have $\kappa(\hat Y) \geq 0$. The non-algebraicity of $Y$ also yields $\kappa(\hat Y) \leq 1$. 
\end{proof}

We describe the structure of $X$ resp. $X'$ more closely. If  the algebraic dimension $a(Y)  = 1,$ we denote by $f: Y \to C$ the
algebraic reduction. 

\begin{corollary} In the setting of Theorem \ref{appli}, the algebraic dimensions  $(a(X),a(Y)) $ can take the following values.
\begin{enumerate} 
\item[(1)] $(0,0),$ and $Y$ is bimeromorphic to a K3-surface or a torus.
\item[(2)] $(1,0),$ and $X$ is bimeromorphic (up to an \'etale quotient possibly)  of a product of a K3-surface or a tori with an elliptic curve.
\item[(3)] $(1,1),$ and $f \circ \varphi$ is the algebraic reduction of $X';$ the general fiber
$f \circ \varphi$ being a ruled surface $\mathbb P(V)$ over a (possibly varying) elliptic curve $E$ with $V = \sO_E \oplus L$ and $L \equiv 0$ but not torsion 
or $V$ the non-split extension of two trivial line bundles. 
\item[(4)] $(2,1),$ and $X' \simeq Y \times \mathbb P_1,$ possibly after a base change $\tilde C \to C.$ 
\end{enumerate}
\end{corollary} 

\begin{proof} Assertion (1), (2)  and (3) are contained in \cite[14.1]{Fuj87}, \cite[9.1]{CP00}. 
As to (4), since it is shown in  \cite[9.1]{CP00} that a more precise structure property holds: 
after passing to $\tilde X$ and $\tilde Y$  via a finite (ramified) cover $\tilde C \to C$, we obtain 
$$\tilde X  \simeq  Z \times_{\tilde C }\tilde Y,$$
where $Z \to \tilde C$ is bimeromorphic to $\Bbb P_1\times \tilde C$, hence the claim.

\end{proof}

\begin{example} The cases $(2)$ and $(4)$ can obviously be realised by products, we  give examples for the other cases. 

a) Let $Y$ be a K3 surface with ${\rm Pic}(Y) = 0$ and set $X = \mathbb P(T_Y).$ Then $X$ does not contain any divisor, hence 
$a(X) = a(Y) = 0.$ 

b) Let $Y$ be a two-dimensional torus of algebraic dimension $1$, and denote by $F$ the general fibre of the algebraic reduction $\holom{f}{Y}{C}$. 
Let $\sL \in {\rm Pic}(Y)$ be a line bundle such that $\sL|_F$ is numerically trivial but not torsion,
and set $X := \mathbb P(\sO_Y \oplus \sL)$. We claim that $a(X) = a(Y) = 1$: 
it suffices to prove that $\kappa(X, \sG) \leq 1$ for any line bundle $\sG$ on $X$.
 Now any line bundle $\sG$ on $X$ is of the form
$$ \sG = \sO_{\mathbb P (\sO_Y \oplus \sL)}(k) \otimes \pi^*(\sF) $$
with a line bundle $\sF$ on $Y.$ If $\kappa (\sG) \geq  1$, then $k \geq 0$ and
$$ 
h^0(X,\sG^{\otimes m}) = h^0(Y,S^{km}(\sO_Y \oplus \sL) \otimes \sF^m)
= \sum_{l=0}^{km} h^0(Y, \sL^l \otimes \sF^m).
$$ 
Since $f$ is the algebraic reduction, the line bundle $\sF$ has degree $0$ on $F$. 
Since $\sL|_F$ is not torsion there exists a unique $l_0 \in \{ 0, \ldots, km\}$ such that $(\sL^l \otimes \sF^m)|_F \simeq \sO_F$.
Thus we have 
$$
h^0(X,\sG^{\otimes m}) 
= h^0(Y, \sL^{l_0} \otimes \sF^m),
$$ 
We immediately obtain $\kappa (X, \sG) \leq 1$. 
\end{example}

\subsection{Threefolds with trivial canonical bundles}

We will study the Albanese map for certain threefolds with terminal singularities (cf. \cite[2.4.1]{BS95} for the existence of Albanese maps
in the presence of rational singularities).

\begin{lemma} \label{lemmaalbanese}
Let $X$ be a (non-algebraic) compact K\"ahler threefold with terminal singularities. If $\kappa (X) = 0$, then 
the Albanese map $\alpha: X \to  A = {\rm Alb}(X)$ is surjective with connected fibres. In particular we have $q(X) \leq 3$.
\end{lemma}

\begin{proof}
Apply \cite[Main Thm.I]{Uen87} to a desingularisation $\hat X \rightarrow X$.
\end{proof}

\begin{theorem}  \label{theoremkxzero}
Let $X$ be a non-algebraic compact K\"ahler threefold with terminal singularities.
If $K_X \equiv 0$, there exists a Galois cover $f: \tilde X \to X$ that is \'etale in codimension one 
such that either $X$ is a torus or a product of an elliptic curve and a K3 surface. 
\end{theorem}

\begin{proof}
By Theorem \ref{theoremmain}, the divisor $K_{X}$ is semi-ample. Thus we can choose $m \in \N$ minimal such that $\sO_{X}(mK_{X}) = \sO_{X}$. 
Let $X' \rightarrow X$ be the induced cyclic cover, then we have $\sO_{X'}(K_{X'}) \simeq \sO_{X'}$.
In particular $X'$ is Gorenstein with terminal singularities \cite[Cor.5.21]{KM98}. Thus the Riemann-Roch formula \ref{propositionRR} gives 
$\chi(X' ,\sO_{X'}) = 0$. 
Since $h^2(X', \sO_{X'}) \geq 1,$ the variety $X'$ being non-algebraic with rational singularities only, we conclude that
$h^1(X', \sO_{X}) \geq 1$. 

Now consider the Albanese map $\alpha: X' \to {\rm Alb}(X')=:A$. By Lemma \ref{lemmaalbanese} the morphism $\alpha$ is surjective with connected fibres. 

{\em Case 1: $q(X')=3$.}  Then $\alpha$ is bimeromorphic. Since $A$ is smooth, $K_{X'} \equiv E$ with $E$ an effective divisor
whose support equals the $\alpha$-exceptional locus. Since $K_{X'} \equiv 0$ we conclude that $X' \simeq A$.

{\em Case 2: $q(X')=2$.} By \cite[Main Thm.I,2)]{Uen87} there exists a finite set $S \subset A$ such that $\alpha$ is an elliptic bundle
over $A \setminus S$. By \cite[Prop.6.7(i)]{CP00} this implies that (after finite \'etale base change) the fibration
$\alpha$ is bimeromorphic to a compact K\"ahler manifold $X''$ that is an elliptic bundle $X'' \rightarrow A$. 
Yet a K\"ahler manifold which is an elliptic bundle over a torus is an \'etale quotient of a torus \cite{Bea83},
so after finite \'etale cover we have $q(X')=q(X'')=3$. We conclude by applying Case 1.

{\em Case 3: $q(X')=1$.} By \cite[Main Thm.I,3)]{Uen87} there exists an analytic fibre bundle $X^* \rightarrow A$ that is bimeromorphic
to $X'$. Since the Kodaira dimension is an invariant of varieties with terminal singularities, we have $\kappa(X^*)=0$. In particular the general
fibre $F$ of $X^* \rightarrow A$ is the blow-up of a torus or a K3 surface\footnote{If $F$ was bimeromorphic to an Enriques or bielliptic
surface we would have $0 = H^2(X^*, \sO_{X^*}) = H^2(X', \sO_{X'})$, a contradiction.}.
We next run a relative MMP over the elliptic curve $A$. Since $X^* \rightarrow A$ is a fibre bundle, every step of this MMP 
is the blow-up along an \'etale multisection of the fibration, so the outcome is an analytic fibre bundle $\bar X \rightarrow A$ such that
the general fibre $F$ is a torus or K3 surface. If $F$ is covered by a torus, then $\bar X$ is a torus after finite \'etale cover, see e.g. \cite{Bea83}. 
If $F$ is a K3 surface, the fibre bundle trivialises after finite \'etale base change $A' \rightarrow A$ (\cite[Cor.4.10]{Fuj78}, cf. \cite[Lemma 2.15]{a19} for more details).
Thus we have $\bar X \simeq A' \times F$. In conclusion we see that (up to finite \'etale cover) we have a bimeromorphic map
$$
\mu : X' \dashrightarrow \bar X
$$
with $\bar X$ a torus or a product  $A' \times F$. Since both $K_{X'}$ and $K_{\bar X}$ are numerically trivial, we see that $\mu$ is an isomorphism in
codimension one \cite{Han87}, \cite[4.3]{Kol89}. Moreover $\mu$ decomposes into a finite sequence of flops by \cite[4.9]{Kol89}. Note however that the last
flop of this sequence yields a rational curve in $\bar X$ that is very rigid (in the sense of \cite[Defn.4.3]{HP13a}). Since $\bar X$ is a torus or a product  $A' \times F$
such a curve does not exist on $\bar X$. Thus $\mu$ is an isomorphism.
\end{proof}

Using the existence of minimal models for a smooth compact K\"ahler threefold, we deduce

\begin{corollary} Let $X$ be a non-algebraic compact K\"ahler threefold with $\kappa (X) = 0.$ There exists a finite cover which is  bimeromorphic to a torus or
a  product of an elliptic curve and a K3 surface. 
\end{corollary} 

It was known since some time that Theorem \ref{theoremkummer} is a consequence of the existence of good minimal models \cite[p.731]{Pet98}. We will now derive Theorem \ref{theoremkummer} from the more general Theorem \ref{theoremkxzero}:

\begin{proof}[Proof of Theorem \ref{theoremkummer}] Since $X$ is not uniruled, $K_X$ is pseudo-effective \cite{Bru06}.
By \cite[Thm.1.1]{HP13a} there exists a minimal model $X \dashrightarrow X'$. Since $X$ has algebraic dimension zero,
we see that $\kappa(X)=0$. Thus we have $\nu(X')=\kappa(X')=0$ by Theorem \ref{theoremmain}, i.e. the canonical divisor $K_{X'}$ 
is numerically trivial. Since $X$ (and hence $X'$) is not covered by curves Theorem \ref{theoremkxzero} yields that 
$X' \simeq T/G$ with $T$ a torus and $G$ a finite group. Since $X$ (and hence $X'$) is not covered by positive-dimensional subvarieties,
the torus $T$ has no positive-dimensional subvarieties. In particular $T/G$ has no positive-dimensional
subvarieties, so $X \dashrightarrow T/G$ extends to a morphism.
\end{proof} 

\newcommand{\etalchar}[1]{$^{#1}$}
\def\cprime{$'$}

\end{document}